\documentclass[a4paper, 11pt]{article}
\usepackage[utf8]{inputenc}
\usepackage[T1]{fontenc}
\usepackage{textcomp}
\usepackage{amsmath, amssymb}
\usepackage{amsthm}
\usepackage[hmargin= 2.5 cm]{geometry}
\usepackage[colorlinks=true]{hyperref}
\usepackage[french]{babel}
\usepackage{mathrsfs}
\title{On the growth of high Sobolev norms for certain one-dimensional Hamiltonian PDEs}
\author{Joseph \bsc{Thirouin}\\ joseph.thirouin@math.u-psud.fr}
\date{\today}

\DeclareMathOperator{\supp}{supp}
\DeclareMathOperator{\sgn}{sgn}

\newtheorem{thm}{Theorem}
\newtheorem{prop}{Proposition}[section]
\newtheorem{lemme}[prop]{Lemma}
\newtheorem{cor}[prop]{Corollary}
\theoremstyle{definition}
\newtheorem*{déf}{Definition}
\theoremstyle{remark}
\newtheorem{rem}{Remark}

\usepackage{enumitem}
\setenumerate[1]{label=(\roman*)}

\begin{document}

\renewcommand{\refname}{Bibliography}
\renewcommand{\abstractname}{Abstract}
\maketitle

\begin{abstract}
This paper is devoted to the study of large time bounds for the Sobolev norms of the solutions of the following fractional cubic Schrödinger equation on the torus :
\[i \partial_t u = |D|^\alpha u+|u|^2 u, \quad u(0, \cdot)=u_0,\]
where $\alpha$ is a real parameter. We show that, apart from the case $\alpha = 1$, which corresponds to a half-wave equation with no dispersive property at all, solutions of this equation grow at a polynomial rate at most. We also address the case of the cubic and quadratic half-wave equations.\par
\textbf{MSC 2010 :} 37K40, 35B45.\par
\textbf{Keywords :} Hamiltonian systems, fractional nonlinear Schrödinger equation, nonlinear wave equation, dispersive properties.
\end{abstract}

\section{Introduction}
In the study of Hamiltonian partial differential equations, understanding the large time dynamics of solutions is an important issue. In usual cases, the conservation of the Hamiltonian along trajectories enables to control one Sobolev norm of the solution (in the so-called \emph{energy space}), but when solutions are globally defined and regular, higher norms could grow despite the conservation laws, reflecting an energy transfer to high frequencies. Even for notorious equations, such as the nonlinear Schrödinger equation on manifolds, it is an old problem to know whether such instability occurs \cite{BourGAFA}, and often still an open question.

We start from the particular case of the defocusing Schrödinger equation on the torus of dimension one, with a cubic nonlinearity :
\begin{equation} \label{schr} i \partial_t u = -\partial_x^2 u+|u|^2u, \quad u(0, \cdot)=u_0.
\end{equation}
Here, $u$ is a function of time $t\in \mathbb{R}$ and of space variable $x\in \mathbb{T}$, and $u_0\in H^1(\mathbb{T})$. Equation \eqref{schr} is Hamiltonian, and because of the energy conservation, its trajectories are bounded is $H^1(\mathbb{T})$. But it is also well-known that \eqref{schr} is integrable (see \cite{KappG}, \cite{zsh}), with conservation laws ensuring that if $u_0$ belongs to $H^s(\mathbb{T})$ for some $s\in \mathbb{N}\backslash \{0\}$, then the solution $u$ remains bounded in $H^s$ (this is even true for any real $s\geq 1$ \cite{KST}).

In order to track down large time instability for Hamiltonian systems, Majda, McLaughlin and Tabak \cite{MMT} suggested to replace the Laplacian in \eqref{schr} by a whole family of pseudo-differential operators : the operators $|D|^\rho$ (sometimes written as $(\sqrt{-\partial_x^2})^\rho$) for real $\rho$. Recall that if $w=\sum_{k\in \mathbb{Z}} w_ke^{ikx}$ is a function on the torus, then
\[ |D|^\rho w= \sum_{k\in \mathbb{Z}} |k|^\rho w_ke^{ikx}. \]

So we consider the following fractional Schr\"odinger equation :
\begin{equation} \label{mmt} i \partial_t u = |D|^\alpha u+|u|^2 u, \quad u(0, \cdot)=u_0,
\end{equation}
where $\alpha$ is any positive number. If $\alpha = 2$, we recognize the classical Schr\"odinger equation \eqref{schr}. In the case $\alpha =1$, \eqref{mmt} is a non-dispersive equation, called the (defocusing) "half-wave" equation :
\begin{equation} \label{hw} i \partial_t u = |D| u+|u|^2 u, \quad u(0, \cdot)=u_0.
\end{equation}
This half-wave equation has been studied by Gérard and Grellier in \cite{aPDE}. In particular, they showed that the dynamics of \eqref{hw} is related to the behaviour of the solutions of a toy model equation, called the cubic Szeg\H{o} equation :
\begin{equation} \label{sz} i \partial_t u = \Pi_+(|u|^2 u), \quad u(0, \cdot)=u_0,
\end{equation}
where $\Pi_+:= \text{\textbf{1}}_{D\geq 0}$ is the projection onto nonnegative Fourier modes. In a more precise way, equation \eqref{sz} appears to be the completely resonant system of \eqref{hw}.

All the equations \eqref{mmt} derive from the Hamiltonian $\mathcal{H}_\alpha (u):=\frac{1}{2}(|D|^\alpha u,u)+\frac{1}{4}\|u\|_{L^4}^4$ for the symplectic structure endowed by the form $\omega (u,v)= \Im m (u,v)$, where $(u,v):=\int_\mathbb{T} u\bar{v}$ denotes the standard inner product on $L^2(\mathbb{T})$. The functional $\mathcal{H}_\alpha $ is therefore conserved along trajectories. Gauge invariance as well as translation invariance also imply the existence of two other conservation laws for equation \eqref{mmt} :
\begin{align*}
Q(u)&:= \frac{1}{2}\|u\|_{L^2}^2 \\
M(u)&:= (Du,u), \quad \text{where } D:=-i\partial_x ,
\end{align*}
\textit{i.e.} the mass and the momentum respectively. Starting from these observations, it has been proved that for $\alpha =1$, equation \eqref{hw} admits a globally defined flow in $H^s$ with $s\geq \frac{1}{2}$ (see \cite{aPDE}). In the case of the half-wave equation, the Brezis-Gallouët inequality \cite{BrG} also ensures that $H^s$-norms of solutions grow at most like $e^{\exp B|t|}$, for some constant $B>0$ depending on $s$ and on the initial data.

The question of the large time instability of global solutions of \eqref{mmt} thus naturally arises : is it possible to find smooth initial data whose corresponding orbits are not bounded in some space $H^s$, or at least not polynomially bounded\footnote{We say that a solution $t\mapsto u(t)$ is \emph{polynomially bounded} in $H^s$ if there are positive constants $C$ and $A$ (not depending on time) such that for all $t\in \mathbb{R}$, $\|u(t)\|_{H^s}\leq C(1+|t|)^A$.} ?

The cubic Szeg\H{o} equation discloses this kind of instability, as recently shown in \cite{IHES}, \cite{livrePG} : for generic smooth initial data, the corresponding solution of the Szeg\H{o} equation in $H^s$ is polynomially unbounded, for any $s>\frac{1}{2}$. Therefore it is reasonable to think that the same statement should hold for the half-wave equation \eqref{hw}, though such a result seems far beyond our reach at this point. Nevertheless the theorem we prove in this paper gives an a priori bound for all solutions of the half-wave equation :

\begin{thm}
\label{theohw}
Let $u_0 \in \mathcal{C}^\infty(\mathbb{T})$, and $t\mapsto u(t)$ the solution of the half-wave equation \eqref{hw} such that $u(0)=u_0$. Given any integer $n\geq 0$, there exist real constants $B$, $C>0$ such that $\forall t\in \mathbb{R}$,
\begin{equation}\label{est2}\|u(t)\|_{H^{1+n}} \leq C e^{B|t|^2}.
\end{equation}
Here, $C$ depends on $n$ and $\|u_0\|_{H^{1+n}}$, whereas $B$ can be chosen equal to $B_n\|u_0\|_{H^{1/2}}^8$, with $B_n >0$ depending only on $n$.
\end{thm}

The bound appearing in \eqref{est2} is an improvement the "double exponential bound" mentioned above. But as a matter of fact, finding any explicit non-trivial solution of \eqref{mmt} is still an open problem, and nothing is known about the optimality of \eqref{est2}. Solution with rapidly growing $H^s$-norms could perfectly well exist. H. Xu \cite{Xu} typically proved the existence of exponentially growing solutions for a perturbation of the Szeg\H{o} equation. See also the result of Hani--Pausader--Tzvetkov--Visciglia \cite{HPTV}, in the context of the Schrödinger equation, as well as its recent counterpart in \cite{Xu2}.

Notice that in the case of the Szeg\H{o} equation, the best bound quantifying the growth of Sobolev norms of solutions is $e^{B|t|}$ : it is obtained by G\'erard and Grellier in \cite[section 3]{ann}. Hence \eqref{est2} is likely to be improved, but recall that, as far as we know, the only way of proving the simple exponential bound for Szeg\H{o} solutions makes use of the Lax pair structure associated with the equation. Elementary methods would only give an $e^{B|t|^2}$ bound (see Appendix \ref{sz-quad}). Unfortunately, such a Lax pair structure apparently does not exist as regards the half-wave equation.

Even so, the proof of \eqref{est2} in theorem \ref{theohw} suggests that we could get a simple exponential bound, instead of $e^{B|t|^2}$, if we could deal with a quadratic nonlinearity. Simply putting an $L^3$-norm in the energy $\mathcal{H}_1$, instead of the $L^4$ one, would give rise to a nonlinearity of the form $|u|u$, and the singularity at the origin may lead to solutions less regular than their initial data. It is possible to avoid this phenomenon, by consider a system of two equations rather than a single scalar equation :
\begin{equation}\label{quad} \left\lbrace\begin{aligned}
i\partial_t u_1 &= |D|u_1+u_2\overline{u_1},\\
i\partial_t u_2 &= |D|u_2+\frac{u_1^2}{2},
\end{aligned}\right.
\end{equation}
with $\left. (u_1,u_2)\right|_{t=0} = (u_1^0,u_2^0)$. System \eqref{quad} only involves (analytic) quadratic nonlinearities. It happens that Schrödinger systems of that kind frequently appear in physics : they are closely linked with the SHG (Second-Harmonic Generation) theory in optics, and the study of propagation of solitons in so-called $\chi ^{(2)}$ (or quadratic) media or materials (for a review, see e.g. \cite[section 4]{kiv-quad}). Quadratic systems are also relevant in fluid mechanics, to describe the interaction between long nonlinear waves in fluid flows \cite{gottwald}. From a mathematical point of view, interest on quadratic systems is more recent \cite{hayashi}.

For the case of system \eqref{quad}, we prove the following theorem :

\begin{thm}\label{theo^2}
Let $(u_1^0,u_2^0) \in \mathcal{C}^\infty(\mathbb{T})\times \mathcal{C}^\infty(\mathbb{T})$, and $t\mapsto (u_1(t),u_2(t))$ the solution of \eqref{quad} such that $(u_1(0),u_2(0))=(u_1^0,u_2^0)$. Given any integer $n\geq 0$, there exist real constants $B'$, $C>0$ such that $\forall t\in \mathbb{R}$,
\begin{equation}\label{est3}\|u_1(t)\|_{H^{1+n}}\text{, }\|u_2(t)\|_{H^{1+n}} \leq C e^{B'|t|}.
\end{equation}
Here, $C$ depends on $n$ and on the sum $\|u_1^0\|_{H^{1+n}}+\|u_2^0\|_{H^{1+n}}$, whereas $B'$ can be chosen equal to $B'_n(\|u_1^0\|_{H^{1/2}}^2+\|u_2^0\|^2_{H^{1/2}})$, with $B'_n >0$ depending only on $n$.
\end{thm}

Let us now return to equation \eqref{mmt}. When $\alpha \neq 1$, \eqref{mmt} has dispersive properties. Using them for $\alpha >1$ and proving some Strichartz estimate for the operator $e^{-it|D|^\alpha}$, Demirbas, Erdo\u{g}an and Tzirakis show in \cite{Demir} that \eqref{mmt} is globally well-posed in the energy space $H^{\frac{\alpha}{2}}$ (and even below). Their method rely on Bourgain's high-low frequency decomposition, but it does not say anything about the possible growth of solutions.

Still for $\alpha >1$, a naive calculation leads to an exponential bound for $H^s$-norms of solutions, \textit{i.e.} a bound of the form $e^{A|t|}$, but results such as Bourgain's \cite{Bourgain} or Staffilani's \cite{Staffilani} suggest that because of dispersion, solutions should be polynomially bounded. On the other hand, the polynomial growth of solutions of \eqref{mmt} for $\alpha >1$ is announced to be true in the work by Demirbas--Erdo\u{g}an--Tzirakis. Indeed we establish a theorem also involving (part of) the case $\alpha <1$ :

\begin{thm}
\label{theommt}
Let $\alpha \in (\frac{2}{3},1)\cup (1,2)$, and $u_0 \in \mathcal{C}^\infty(\mathbb{T})$. There exist a unique $u\in \mathcal{C}^\infty(\mathbb{R},\mathcal{C}^\infty(\mathbb{T}))$ solution of \eqref{mmt} with $u(0)=u_0$.
Furthermore, given any integer $n\geq 0$, there exist real constants $A$, $C>0$, and a constant $C'>0$ depending only on $\|u_0\|_{H^{\alpha +n}}$, such that $\forall t\in \mathbb{R}$,
\begin{equation}\label{est1}\|u(t)\|_{H^{\alpha +n}} \leq C'(1+C|t|)^A.
\end{equation}

When $\alpha \in (1,2)$, we can choose
\[ A=\frac{2n+\alpha}{\alpha -1}, \quad C=C_{\alpha ,n}\|u_0\|_{H^{\alpha /2}}^{4+\frac{\alpha -1}{2n+\alpha}},\]
where $C_{\alpha ,n}>0$ is a constant which only depends on $\alpha$ and $n$.

When $\alpha \in (\frac{2}{3},1)$, a possible choice is
\[ A=\frac{2n(23\alpha-2)}{(2\alpha -1)(3\alpha -2)}+\frac{10\alpha}{3\alpha -2}.\]
\end{thm}

Here again, it is not known whether \eqref{est1} is optimal or not. When $\alpha >1$, Demirbas \cite{demirbas2014almost} found, by a probabilistic way inspired by the works of Bourgain \cite{BourComm}, solutions growing at most like a power of $\log (1+|t|)$. But in any case, proving that some solutions of \eqref{mmt} do blow up for large time, even at a very low rate, would be a big step forward.

Combining theorem \ref{theohw} and \ref{theommt} thus indicate that $\alpha =1$ is an isolate point in the family of equations \eqref{mmt}. Notice that, when $\alpha <1$, theorem \ref{theommt} includes the existence of a flow, which had not been proved so far. As for the condition $\alpha >\frac{2}{3}$, it appears to be convenient in the proof for technical reasons ; but since the heart of our work is to prove that the case $\alpha =1$ is more likely to disclose weak turbulence phenomena than other cases, we postpone discussions and comments concerning the relevance of the value $\frac{2}{3}$ until Appendix \ref{wellp}.

The proof of theorems \ref{theohw}, \ref{theo^2} and \ref{theommt} is based on an idea developped by Ozawa and Visciglia in \cite{OzV} : the authors introduce a \emph{modified energy method}, in order to sharpen $H^1$-estimates and thus prove well-posedness for the half-wave equation with quartic nonlinearity. To put it shortly, their idea is to introduce a nonlinear energy which is in fact a perturbation of the norm they wish to bound. The perturbation does not modify the size of the norm, but induces simplifications while differentiating, so that time-differentiation behaves like a self-adjoint operator.

In the sequel of this paper, we begin by proving theorems \ref{theohw}, \ref{theo^2}, and the first part of theorem \ref{theommt}, with elementary tools. Then we address the case of $\alpha <1$, using Bourgain spaces as in \cite{BGT}, which we fully develop for the convenience of the reader.

We were about to finish this paper when we were informed of a work by Planchon and Visciglia also applying the modified energy method to solutions of nonlinear Schrödinger equations on certain Riemaniann manifolds, and for every power nonlinearity.

The author would like to express his gratitude towards P. Gérard for his deep insight and generous advice. He also thanks J.-C. Saut for the references concerning quadratic medias.

\section{The case $\alpha \geq 1$}

Throughout this section, we suppose that $\alpha \in [1,2)$. We fix $u_0 \in \mathcal{C}^\infty(\mathbb{T})$, and we study $t\mapsto u(t)$, the associated solution of \eqref{mmt} (or \eqref{hw}, depending on the value of $\alpha$).

\subsection{The modified energy method}\label{ozv}
Fix $n\in \mathbb{N}$. The following lemma gathers some standard inequalities of which we will make an extensive use.
\begin{lemme}\label{baba}
There exist an absolute constant $C>0$, a constant $C_\alpha$ depending on $\alpha$ and a constant $C_n$ depending on $n$ such that, for every $t\in \mathbb{R}$,
\begin{enumerate}
\item \label{alpha/2}$\|u(t)\|_{H^{\alpha /2}} \leq C \|u_0\|_{H^{\alpha /2}}^2$,
\item \label{sobo}if $\alpha > 1$, $\|u(t)\|_{L^\infty} \leq C_\alpha \|u_0\|^2_{H^{\alpha /2}}$,
\item \label{brezg}if $\alpha = 1$, $\|u(t)\|_{L^\infty} \leq C_n \|u_0\|_{H^{1 /2}}^2 \sqrt{\log \left( 1+\dfrac{\|u(t)\|_{H^{1+n}}^2}{C^2\|u_0\|_{H^{1/2}}^4} \right)}$.
\end{enumerate}
\end{lemme}

We justify briefly these inequalities : \ref{alpha/2} derives from the conservation of $Q+\mathcal{H}_\alpha$ together with the Sobolev embeddings in dimension one, and \ref{sobo} is a consequence of \ref{alpha/2} and the injection $H^{\alpha /2}\hookrightarrow L^\infty$. As for \ref{brezg}, it follows from the classical Brezis-Gallouët inequality : for $s>\frac{1}{2}$, and $w\in H^s(\mathbb{T})$,
\begin{equation}\label{brezg-init}
\|w\|_{L^\infty} \leq C_s \|w\|_{H^{1/2}}\left[ \log \left(1+\frac{\|w\|_{H^s}}{\|w\|_{H^{1/2}}} \right) \right] ^{1/2}
\end{equation}
(see e.g. \cite{ann}). Here, to infer \ref{brezg}, we begin by squaring the ratio of the two norms, and we then take into account the fact that the function $x\mapsto x\sqrt{\log (1+\frac{1}{x^2})}$ is increasing.

\vspace{1em}
To prove the estimates we have in mind, we are going to establish an inequality between the $H^{\alpha +n}$-norm of the solution and its derivative, and apply a Gronwall lemma. As announced, we define for this purpose a well-chosen nonlinear functional\footnote{From now on, the time-dependence of the terms will always be implicit. In addition, we will always restrict ourselves to nonnegative times $t\geq 0$, since it is possible to reverse the evolution of \eqref{mmt} via the transformation $u(t) \leftrightarrow \bar{u}(-t)$.} :
\begin{equation}\label{energy}
\mathcal{E}_{\alpha ,n}(u) := \|u\|_{L^2}^2 + \||D|^{\alpha +n}u\|_{L^2}^2
+ 2\Re e(|D|^{\alpha +n}u, |D|^{n}(|u|^2u)) -\textstyle\frac{1}{2} \displaystyle\||D|^{\frac{\alpha}{2}+n}(|u|^2)\|_{L^2}^2.
\end{equation}
Roughly speaking, $\mathcal{E}_{\alpha ,n}$ is a perturbation of the square of the $H^{\alpha +n}$-norm of $u$ (the first two terms) by the means of two corrective quantities. First of all, let us show that the latter do not substantially modify the size of $\|u\|_{H^{\alpha +n}}^2$. To turn this into a rigorous statement, we begin by restricting ourselves to intervals of time on which $\|u\|_{H^{\alpha +n}}$ is larger than a certain constant $\mathcal{M}$ depending on $\|u_0\|_{H^{\alpha /2}}$, and we show that on such intervals,
\begin{equation}\label{ordreE} \frac{1}{2}\|u\|_{H^{\alpha +n}}^2 \leq \mathcal{E}_{\alpha ,n} (u) \leq 2\|u\|_{H^{\alpha +n}}^2,
\end{equation}
for a suitable choice of $\mathcal{M}$ which we precise later.

Set $J_1(u):=2\Re e(|D|^{\alpha +n}u, |D|^{n}(|u|^2u))$. We can write\footnote{The symbol $\lesssim$ is understood as refering to constants depending only on $n$, or absolute constants, whose explicit form is not particularly meaningful.}
\begin{align*}
|J_1(u)| &\lesssim \||D|^{\alpha +n}u\|_{L^2}\||D|^{n}(|u|^2u)\|_{L^2}
\\& \lesssim \|u\|_{H^{\alpha +n}}\|u\|_{L^\infty}^2\|u\|_{H^{n}},
\end{align*}
where we used the tame estimates for products in $H^s$, for $s\geq 0$. Then interpolate $H^{n}$ between $H^{\alpha +n}$ and $H^{\alpha /2}$ (or just bound the $L^2$-norm by a constant if $n=0$), and using lemma \ref{baba}, get
\[ |J_1(u)| \lesssim \|u\|_{H^{\alpha +n}}^{2-\varepsilon_{\alpha ,n}}\|u_0\|_{H^{\alpha /2}}^{2\varepsilon_{\alpha ,n}}\|u\|_{L^\infty}^2 , \quad \text{where }\varepsilon_{\alpha ,n}:= \min \left( 1, \frac{2\alpha}{2n+\alpha}\right) . \]
On the other side, introducing $J_2(u):=-\||D|^{\frac{\alpha}{2}+n}(|u|^2)\|_{L^2}^2$, we similarly obtain :
\[ |J_2(u)| \lesssim \|u\|_{L^\infty}^2\|u\|^2_{H^{\frac{\alpha}{2}+n}} \lesssim \|u\|_{H^{\alpha +n}}^{2-\varepsilon_{\alpha ,n}}\|u_0\|_{H^{\alpha /2}}^{2\varepsilon_{\alpha ,n}}\|u\|_{L^\infty}^2. \]
In sight of \ref{sobo} and \ref{brezg}, all these estimates show that $J_1(u)$ and $J_2(u)$ are of lower order than $\|u\|_{H^{\alpha +n}}^2$. More precisely, if we now set, for real $x$,
\[ g_\alpha(x) := \left\lbrace \begin{aligned} &x^{\varepsilon_{\alpha ,n}/2} &\quad \text{if }\alpha >1, \\ &\frac{x^{\varepsilon_{1 ,n}/2}}{\log \left(1+\frac{x^2}{C^2\|u_0\|^4_{H^{1/2}}}\right) } &\quad \text{if }\alpha =1,\end{aligned} \right.\]
we see that it suffices to request for instance that $g_\alpha (\|u\|_{H^{\alpha +n}})\gg\|u_0\|_{H^{\alpha /2}}^{4+2\varepsilon_{\alpha ,n}}$, which holds true whenever $\|u\|_{H^{\alpha +n}}$ is greater than a certain $\mathcal{M}$. As a conclusion, \eqref{ordreE} is proved on intervals of the form $[T^*,T]$ which satisfy
\[\|u(t)\|_{H^{\alpha +n}}\geq \tilde{\mathcal{M}}:=\max\{ \mathcal{M}, 2\|u_0\|_{H^{\alpha +n}}\}, \: \forall t\in [T^*,T],\quad\text{and}\quad \|u(T^*)\|_{H^{\alpha +n}}=\tilde{\mathcal{M}}.\]

Now we study the evolution of $\mathcal{E}_{\alpha ,n} (u)$ on $[T^*,T]$. As the $L^2$-norm of $u$ is conserved, we denote by $J_0(u):= \||D|^{\alpha +n}u\|_{L^2}^2$, and compute at once :
\[\frac{d}{dt}J_0(u)=2\Re e (|D|^{\alpha +n}\dot{u}, |D|^{\alpha +n}u),\]
where the dot refers to the time-derivative, and commutes with $|D|^\rho$ for all $\rho$. According to equation \eqref{mmt}, $|D|^\alpha u=i\dot{u}-|u|^2u$, so
\[\frac{d}{dt}J_0(u)=2\Im m (|D|^{\alpha +n}\dot{u}, |D|^n\dot{u})-2\Re e (|D|^{\alpha +n}\dot{u}, |D|^{n}(|u|^2u)).\]
Because of the imaginary part, the first term of this sum is zero. As for the second one, it combines with the time-derivative of $J_1(u)$, and we thus have
\[ \begin{aligned}
\frac{d}{dt}[J_0+J_1](u)&=2\Re e (|D|^{\alpha +n}u, |D|^{n}(|u|^2u)\dot{\:})
\\ &=2\Re e \left( \partial_x^{n}|D|^\alpha u, \partial_x^{n}(\dot{u}|u|^2+ u[|u|^2]\dot{\:})\right) .
\end{aligned}\]
Applying Leibniz formula, we get three terms (or only two when $n=0$) :
\begin{align*} 2&\Re e (\partial_x^{n}|D|^\alpha u, (\partial_x^{n}\dot{u})|u|^2)
\\ &+2\Re e \left( \partial_x^{n}|D|^\alpha u, \sum_{k=1}^{n}\binom{n}{k}[(\partial_x^{n-k}\dot{u})\partial_x^k(|u|^2)+(\partial_x^ku)\partial_x^{n-k}(|u|^2)\dot{\:} \right)
\\ &+2\Re e (\partial_x^{n}|D|^\alpha u, u\partial_x^{n}(|u|^2)\dot{\:}).
\end{align*}
Each of these terms has to be estimated. The first one and the third one are more tricky, since all the time- and space-derivative are concentrated on the same function.

\underline{First term} : A simplification fortunately occurs. Rewrite
\[\partial_x^{n}|D|^\alpha u= i\partial_x^{n}\dot{u}-\partial_x^{n}(|u|^2u),\]
and observe that $\Re e (i\partial_x^{n}\dot{u}, (\partial_x^{n}\dot{u})|u|^2)=0$. The first term then equals $-2\Re e (\partial_x^{n}(|u|^2u), (\partial_x^{n}\dot{u})|u|^2)$. Let $Q_1$ be this new quantity. Assuming that $n\geq 1$, we can bound
\[ |Q_1| \lesssim \||u|^2u\|_{H^{n}}\|u\|_{L^\infty}^2\|\dot{u}\|_{H^{n}} \lesssim \|u\|_{L^\infty}^4\|u\|_{H^{n}}(\|u\|_{H^{\alpha +n}}+\|u\|_{L^\infty}^2\|u\|_{H^n}).\]
Indeed, because of the equation, we have $\|\dot{u}\|_{H^s} \lesssim \|u\|_{H^{\alpha +s}}+\|u\|_{L^\infty}^2\|u\|_{H^s}$ for any $s\geq 0$, so that $\|\dot{u}\|_{H^n}\lesssim\|u\|_{H^{\alpha +n}}$ (using again the property of the interval $[T^*,T]$). Hence $|Q_1|\lesssim \|u\|_{H^{\alpha +n}}^{2-\varepsilon_{\alpha ,n}}\|u_0\|_{H^{\alpha /2}}^{2\varepsilon_{\alpha ,n}}\|u\|_{L^\infty}^4$ with the same $\varepsilon_{\alpha ,n}$ as above (notice that it is true even if $n=0$).

\underline{Second term} : As announced, we suppose here that $n\geq 1$, and fix a $k\in \{ 1, \cdots ,n\}$. We must estimate $Q_2^{(k)}:=2\Re e ( \partial_x^{n}|D|^\alpha u, (\partial_x^{n-k}\dot{u})\partial_x^k(|u|^2))$. Using the Sobolev embedding $H^{1/4} \hookrightarrow L^4$ as well as tame estimates again, write
\[ \begin{aligned} |Q_2^{(k)}| &\lesssim \|\partial_x^{n}|D|^\alpha u\|_{L^2}\|\partial_x^{n-k}\dot{u}\|_{L^4}\|\partial_x^k(|u|^2)\|_{L^4}
\\ &\lesssim \|u\|_{H^{\alpha +n}}\|u\|_{H^{\alpha + n-k+\frac{1}{4}}}\|u\|_{L^\infty}\|u\|_{H^{k+\frac{1}{4}}}.
\end{aligned}\]
Interpolate the $H^s$-norms between $H^{\alpha /2}$ et $H^{\alpha +n}$, and get finally
\[|Q^{(k)}_2| \lesssim \|u\|_{H^{\alpha +n}}^{2-\theta( \alpha, n)}\|u_0\|_{H^{\alpha /2}}^{2+2\theta (\alpha ,n)}\|u\|_{L^\infty},\]
where $\theta (\alpha ,n):=\frac{\alpha -1}{2n +\alpha}$.

In the same way, $Q'^{(k)}_2:=2\Re e ( \partial_x^{n}|D|^\alpha u, (\partial_x^ku)\partial_x^{n-k}(|u|^2)\dot{\:})$ can be proven to be controlled by the same quantity (with the same exponents).

\underline{Third term} : This term is the most delicate. We have
\[ \begin{aligned} 2\Re e (\partial_x^{n}|D|^\alpha u, u\partial_x^{n}(|u|^2)\dot{\:})&= 2\Re e \left(\bar{u}\partial_x^{n}|D|^\alpha u, \partial_x^{n}(|u|^2)\dot{\:}\right)
\\&\simeq 2\Re e \left(\partial_x^{n}(\bar{u}|D|^\alpha u), \partial_x^{n}(|u|^2)\dot{\:}\right) ,
\end{aligned}\]
where the $\simeq$ sign means that the equality is true up to terms of order $\|u\|_{H^{\alpha +n}}^{2-\theta (\alpha ,n)}\|u_0\|_{H^{\alpha /2}}^{2+2\theta (\alpha ,n)}\|u\|_{L^\infty}$ (which we control in the same way as for the second term). Thus we would like to estimate $Q_3:=2\Re e (|D|^{n}(\bar{u}|D|^\alpha u), |D|^{n}(|u|^2)\dot{\:})$, but as $|u|^2$ is real, it appears, expanding the real part, that $Q_3=(|D|^{n}(\bar{u}|D|^\alpha u+u|D|^\alpha \bar{u}), |D|^{n}(|u|^2)\dot{\:})$, which combines with the time-derivative of $J_2(u)$, and finally leads to the following expression :
\begin{equation}\label{leib} \left(|D|^{n}\left[ \bar{u}|D|^\alpha u+u|D|^\alpha \bar{u}-|D|^\alpha (\bar{u}u)\right], |D|^{n}(|u|^2)\dot{\:}\right).
\end{equation}
Now we have a very simple Leibniz lemma on the operator $|D|^\alpha $, which we will prove in section \ref{secleib} :
\begin{lemme}\label{leiblem}
Let $\alpha \in [1,2)$. For any integer $n\in \mathbb{N}$, there is a constant $C_n>0$ depending only on $n$, such that for all function $u \in H^{\alpha +n}(\mathbb{T})$,
\[ \|\bar{u}|D|^\alpha u+u|D|^\alpha \bar{u}-|D|^\alpha (\bar{u}u)\|_{H^{n}} \leq C_n \|u\|_{H^{\alpha /2}}^{1+\frac{\alpha -1}{2n +\alpha}}\|u\|_{H^{\alpha +n}} ^{1-\frac{\alpha -1}{2n +\alpha}}.\]
\end{lemme}
Such a result is better than the crude $L^\infty$-$H^{\alpha +n}$ estimate, because of the exponent of the $H^{\alpha +n}$-norm (which is strictly less than $1$ as soon as $\alpha >1$).

Consequently, expression \eqref{leib} is controlled by $\|u\|_{H^{\alpha +n}}^{2-\theta(\alpha ,n)}\|u_0\|_{H^{\alpha /2}}^{2+2\theta (\alpha ,n)}\|u\|_{L^\infty}$ as well. To sum up, only the second and the third term really matter, whence
\[ \left| \frac{d}{dt}\mathcal{E}_{\alpha ,n} (u) \right| \lesssim \|u\|_{H^{\alpha +n}}^{2-\theta(\alpha ,n)}\|u_0\|_{H^{\alpha /2}}^{2+2\theta (\alpha ,n)}\|u\|_{L^\infty}. \]

\vspace{1em}
First, assume $\alpha >1$. In this situation, we can incorporate the $L^\infty$-norm into the $H^{\alpha /2}$ one (see \ref{sobo} of lemma \ref{baba}), so for $t\in [T^*,T]$,
\[ \left| \int_{T^*}^t \frac{d}{d\tau}\mathcal{E}_{\alpha ,n} (u)d\tau \right| \leq  \int_{T^*}^t \left| \frac{d}{d\tau}\mathcal{E}_{\alpha ,n} (u) \right| d\tau \lesssim \|u_0\|_{H^{\alpha /2}}^{4+2\theta (\alpha ,n)}\int_{T^*}^t \|u(\tau )\|_{H^{\alpha +n}}^{2-\theta (\alpha ,n)} d\tau .\]
Furthermore, remembering our estimates \eqref{ordreE} on $\mathcal{E}_{\alpha ,n} (u)$,
\[ \left| \int_{T^*}^t \frac{d}{d\tau}\mathcal{E}(u)d\tau \right| = |\mathcal{E}(u)(t)-\mathcal{E}(u)(T^*)| \geq \frac{1}{2}\|u(t)\|_{H^{\alpha +n}}^2 - 2\|u(T^*)\|_{H^{\alpha +n}}^2.\]
Let $f(t):=\|u(t)\|_{H^{\alpha +n}}^2$. The above calculation ensures that for some $C_{\alpha ,n}>0$ depending on $\alpha$ and $n$,
\[ f(t) \leq 4f(T^*)+C_{\alpha ,n}\|u_0\|_{H^{\alpha /2}}^{4+2\theta (\alpha ,n)}\int_{T^*}^tf(\tau )^{1-\frac{1}{2}\theta (\alpha ,n)} d\tau .\]
Now $\theta (\alpha ,n)$ is positive. A "Gronwall's lemma" argument (which is also known as "Osgood's lemma") thus proves that $f(t) \leq 4f(T^*)(1+C|t-T^*|)^A$ for $t\in [T^*,T]$. Notice that the value of $f(T^*)$, \emph{i.e.} of $\tilde{\mathcal{M}}$, only depends on the value of $\|u_0\|_{H^{\alpha +n}}$. Moreover, this inequality remains true even for $t$ outside any interval of type $[T^*,T]$, so it globally holds and the first part of theorem \ref{theommt} is proved. At last, the constant $A$ can be set to $\frac{2}{\theta(\alpha ,n)}$, \emph{i.e.} $\frac{4n+2\alpha}{\alpha -1}$, which implies the statement.

\vspace{1em}
It remains to consider the case $\alpha =1$. This time, $\theta (1, n)=0$ for any $n$, and in addition, the $L^\infty$-norm of $u$ is not bounded by a constant anymore. Using part \ref{brezg} of lemma \ref{baba}, and going on as in the previous case with an auxiliary function $g(t)=\|u(t)\|_{H^{1 +n}}^2/C^2\|u_0\|_{H^{1/2}}^4$, we find,
\[ g(t) \leq 4g(T^*)+C_n\|u_0\|_{H^{1/2}}^4\int_{T^*}^tg(\tau )\sqrt{\log (1+g(\tau ))} d\tau ,\]
for all $t\in [T^*,T]$. Osgood's lemma then yields $g(t) \leq Cg(T^*)e^{B|t|^2}$, and the proof of theorem \ref{theohw} is complete.

\subsection{A Leibniz lemma}\label{secleib}
We eventually turn to the
\begin{proof}[Proof of lemma \ref{leiblem}]
Let $u=\sum_{k\in \mathbb{Z}} u_ke^{ikx} \in H^{\alpha +n}(\mathbb{T})$. We intend to control the $H^n$-norm of $F_\alpha(u):=\bar{u}|D|^\alpha u+u|D|^\alpha \bar{u}-|D|^\alpha (\bar{u}u)$. A straightforward computation yields
\[ F_\alpha(u) = \sum_{k=-\infty}^{+\infty} e^{ikx} \left( \sum_{l=-\infty}^{+\infty} \left( |l|^\alpha +|k-l|^\alpha -|k|^\alpha \right) u_l\overline{u_{l-k}}\right).\]
The key idea is to replace $|l|^\alpha +|k-l|^\alpha -|k|^\alpha$ by a more symmetric coefficient, and then to recognize a convolution product. More precisely, define a continuous function $\varphi$ of the real variable $x\in \mathbb{R}\setminus \{0,1\}$ with 
\[ \varphi (x):=\frac{|x|^\alpha +|1-x|^\alpha -1}{|x|^{\frac{\alpha}{2}}|1-x|^{\frac{\alpha}{2}}},\]
and $\varphi (0)=\varphi (1)=0$. For every $l$, $k\in \mathbb{Z}$ with $k\neq 0$, we then have $|l|^\alpha +|k-l|^\alpha -|k|^\alpha= \varphi (\frac{l}{k})|l|^{\frac{\alpha}{2}}|k-l|^{\frac{\alpha}{2}}$, which is true even if $k=0$, once we have assigned $\varphi (\pm \infty )=2$, by convention.

Now, to show that $\varphi $ is bounded on $\mathbb{R}$, it suffices to check that it is bounded near $0$ and $-\infty$, and then to invoke the symmetry of $\varphi$ around $x=\frac{1}{2}$. And actually, since $\alpha \leq 2$, $\lim_{x\to 0}\varphi (x)=0$ and $\lim_{x\to -\infty}\varphi (x)=2$. Studying the variations of $\varphi$ even show that $\|\varphi \|_{L^\infty(\mathbb{R})}=2$, and hence is independent of $\alpha$.

As a consequence,
\[ \begin{aligned} &\sum_{k=-\infty}^{+\infty} |k|^{2n}\left| \sum_{l=-\infty}^{+\infty} (|l|^\alpha +|k-l|^\alpha -|k|^\alpha )u_l \overline{u_{l-k}} \right| ^2 \\
&\leq 2 \sum_{k=-\infty}^{+\infty}|k|^{2n} \left| \sum_{l=-\infty}^{+\infty} |l|^{\frac{\alpha}{2}}|k-l|^{\frac{\alpha}{2}}|u_l||\overline{u_{l-k}}|\right| ^2 \\
&\leq 2^{2n} \sum_{k=-\infty}^{+\infty}\left[ \left| \sum_{l=-\infty}^{+\infty} |l|^{\frac{\alpha}{2}+n}|k-l|^{\frac{\alpha}{2}}|u_l||\overline{u_{l-k}}|\right| ^2 +  \left| \sum_{l=-\infty}^{+\infty} |l|^{\frac{\alpha}{2}}|k-l|^{\frac{\alpha}{2}+n}|u_l||\overline{u_{l-k}}|\right| ^2 \right], \end{aligned}\]
because of the inequality $|k|^n \leq 2^{n-1}(|l|^n+|k-l|^n)$, satisfied for any $n\geq 1$, any $k$ and $l$.

When $n=0$, the result is rather immediate, since
\[ \sum_{k=-\infty}^{+\infty} \left| \sum_{l=-\infty}^{+\infty} |l|^{\frac{\alpha}{2}}|k-l|^{\frac{\alpha}{2}}|u_l||\overline{u_{l-k}}|\right| ^2 = \left\| \left| |D|^{\frac{\alpha}{2}}\tilde{u}\right| ^2\right\| ^2_{L^2}=\left\| |D|^{\frac{\alpha}{2}}\tilde{u}\right\| _{L^4}^4,\]
where $\tilde{u}:=\sum_{k\in \mathbb{Z}} |u_k|e^{ikx}$. Using as always the embedding $L^4 \hookrightarrow H^{1/4}$, and interpolating between $\alpha /2$ and $\alpha$, we get the result.

From here on, we suppose $n\neq 0$. We shall deal with the first part of the above sum (the last one follows identically). We consider two sequences $v:=(|l|^{\frac{\alpha}{2}+n}|u_l|)_{l\in \mathbb{Z}}$ and $w:=(|l|^{\frac{\alpha}{2}}|\overline{u_{-l}}|)_{l\in \mathbb{Z}}$. With these notations,
\[\sum_{k=-\infty}^{+\infty}\left| \sum_{l=-\infty}^{+\infty} |l|^{\frac{\alpha}{2}+n}|k-l|^{\frac{\alpha}{2}}|u_l||\overline{u_{l-k}}|\right| ^2 = \|v\star w\|_{\ell ^2}^2 .\]
By Schur's lemma, $\|v\star w\|_{\ell ^2} \leq \|v\|_{\ell ^2}\|w\|_{\ell ^1}$. But $\|v\|_{\ell ^2}\leq \|u\|_{H^{\frac{\alpha}{2}+n}}$. As for $\|w\|_{\ell ^1}$, write
\[ \begin{aligned}\sum_{k=-\infty}^{+\infty} |w_k| &= \sum_{|k|\leq N} |w_k| + \sum_{|k|> N} |w_k|
\\ &\leq \left(\sum_{|k|\leq N}|w_k|^2\right) ^{\frac{1}{2}}\sqrt{2N+1} + \left(\sum_{|k|> N}|w_k|^2(1+|k|^2)\right) ^{\frac{1}{2}}\left(\sum_{|k|> N}\frac{1}{1+|k|^2}\right) ^{\frac{1}{2}}
\\ &\leq \sqrt{2N+1}\|w\|_{\ell ^2}+\sqrt{\frac{2}{N}}\|w\|_{h^1}.
\end{aligned}\]
Taking the infimum over $N \in \mathbb{N}$, we finally get $\|w\|_{\ell ^1} \leq \sqrt{\|w\|_{\ell ^2}\|w\|_{h^1}}$. In our case, $\|w\|_{\ell ^2}\leq\|u\|_{H^{\frac{\alpha}{2}}}$, and similarly $\|w\|_{h^1}\leq\|u\|_{H^{\frac{\alpha}{2}+1}}$.

Now we can conclude :
\[ \begin{aligned} \|F_\alpha (u)\|_{H^{n}}^2 &\lesssim \|u\|_{H^{\frac{\alpha}{2}+n}}^2\|u\|_{H^{\frac{\alpha}{2}}}\|u\|_{H^{\frac{\alpha}{2}+1}} \\
& \lesssim \|u\|_{H^\frac{\alpha}{2}}^{2\left( 1+\frac{\alpha -1}{2n+\alpha}\right) }\|u\|_{H^{\alpha +n}}^{2\left( 1-\frac{\alpha -1}{2n+\alpha}\right) } ,\end{aligned}\]
which corresponds to the statement.
\end{proof}

\begin{rem}
Lemma \ref{leiblem} probably takes the best advantage of the form of \eqref{leib}. There is still a kind of "virial" identity which holds for any value of $\alpha$ :
\[\frac{d}{dt}|u|^2=i\left(u|D|\bar{u}-\bar{u}|D|u \right) ,\]
but the inequality $\|u|D|\bar{u}-\bar{u}|D|u\|_{L^2}\lesssim \|u\|_{H^{1/2}}\|u\|_{H^1}$, for instance, is false. As a counter-exemple, choose
\[u_N(x)=\dfrac{1}{\sqrt{\log N}}\sum_{n=1}^N\frac{e^{inx}}{n},\]
and let $N\to +\infty$.
\end{rem}

\subsection{Quadratic half-wave equations}
We come to the system of equations \eqref{quad} and the proof of \eqref{est3}. To see the Hamiltonian structure of \eqref{quad}, choose $L^2(\mathbb{T}) \times L^2(\mathbb{T})$ as a phase space, endowed with the inner product $\langle (u_1,u_2) , (v_1,v_2) \rangle := (u_1,v_1)+(u_2,v_2)$. Taking the imaginary part of $\langle \cdot ,\cdot \rangle $ as our symplectic form, we infer from a simple calculation that the Hamiltonian $\tilde{\mathcal{H}}(u_1,u_2):=\frac{1}{2}[(|D|u_1,u_1)+(|D|u_2,u_2) + \int_\mathbb{T}\Re e(u_1^2\overline{u_2})]$ is associated to the system \eqref{quad}.

Notice that the functional $\tilde{\mathcal{H}}$ is invariant under the flow $(u_1,u_2)\mapsto (e^{i\theta}u_1,e^{2i\theta}u_2)$, with $\theta$ varying in $\mathbb{R}$. It follows then from the Noether theorem that $\tilde{\mathcal{Q}}(u_1,u_2):=\|u_1\|_{L^2}^2+2\|u_2\|_{L^2}^2$ is a conservation law for the system \eqref{quad}. As a consequence, the $L^2$-norms of $u_1$, $u_2$ stay bounded along the flow lines. In addition, the conservation of $\tilde{\mathcal{H}}$ as well as $\tilde{\mathcal{Q}}$ claims the uniform boundedness of $\|u_1\|_{H^{1/2}}$ and $\|u_2\|_{H^{1/2}}$ with respect to time.

Immediately, we get, for each $n\geq 0$, \[\|u_1\|_{H^{1+n}}, \|u_2\|_{H^{1+n}}\lesssim e^{B|t|^2},\]where $B>0$ is independent of time : this follows from a straightforward application of inequality \ref{brezg} in lemma \ref{baba}. But now, set $F(t):=( \|u_1\|_{H^{1+n}}^2+ \|u_2\|^2_{H^{1+n}})$. Here we won't repeat the details of section \ref{ozv}, but we suppose from the beginning that $F$ is "big enough", and we compute
\[ \begin{aligned} \frac{d}{dt}&\left[ \| |D|^{1+n}u_1\|_{L^2}^2 + \| |D|^{1+n}u_2\|_{L^2}^2\right]\\&= 2\Re e\left[ (|D|^{1+n}\dot{u_1},|D|^{1+n}u_1)+(|D|^{1+n}\dot{u_2},|D|^{1+n}u_2)\right] \\ &= -2\Re e\left[ (|D|^{1+n}\dot{u_1}, |D|^n(u_2\overline{u_1})) + \left( |D|^{1+n}\dot{u_2},|D|^n\left(\frac{u_1^2}{2}\right) \right) \right] . \end{aligned} \]
Then, correcting the initial quantity with terms of lower order than $F(t)$, we rather estimate
\[A:= 2\Re e\left[ (|D|^{1+n}u_1, |D|^n(\dot{u_2}\overline{u_1}+u_2\dot{\overline{u_1}})) + \left( |D|^{1+n}u_2,|D|^n(\dot{u_1}u_1) \right) \right] .\]
Now apply the Leibniz formula :
\begin{equation}\label{quadleib}
\begin{aligned} A=2&\Re e \left[ (\partial_x^{n}|D| u_1, (\partial_x^{n}\dot{u_2})\overline{u_1})+(\partial_x^{n}|D| u_2, (\partial_x^{n}\dot{u_1})u_1) \right]
\\ &+2\Re e \left( \partial_x^{n}|D| u_1, \sum_{k=0}^{n-1}\binom{n}{k}[(\partial_x^{k}\dot{u_2})\partial_x^{n-k}\overline{u_1}+(\partial_x^{n-k}u_2)\partial_x^{k}\dot{\overline{u_1}} \right)
\\ &+2\Re e \left( \partial_x^{n}|D| u_2, \sum_{k=0}^{n-1}\binom{n}{k}(\partial_x^{k}\dot{u_1})\partial_x^{n-k}u_1 \right)
\\ &+2\Re e (\partial_x^{n}|D|u_1, u_2\partial_x^{n}\dot{\overline{u_1}}). \end{aligned}
\end{equation}
We aim at showing that each of these terms is controlled by $F(t)$.

Exactly as before, we have a cancellation occuring in the first line of \eqref{quadleib} : replace $\dot{u_1}$ by $i|D|u_1$, and $\dot{u_2}$ by $i|D|u_2$ (here again the nonlinearities can be neglected), and observe that
\[ 2\Im m \left[ (\partial_x^{n}|D| u_1, (\partial_x^{n}|D|u_2)\overline{u_1})+(\partial_x^{n}|D| u_2, (\partial_x^{n}|D|u_1)u_1) \right] = 0.\]

Concerning the second and the third line of \eqref{quadleib}, straightforward Sobolev estimates and interpolation inequalities are enough to conclude.

As for the fourth line of \eqref{quadleib}, where all time- and space-derivatives concentrate on the same function, we need an equivalent of lemma \ref{leiblem} for the operator $|D|^{1/2}$, namely :

\begin{lemme}[Kenig-Ponce-Vega, see \cite{KPV}] \label{KPV}
For $f$, $g : \mathbb{T}\to \mathbb{C}$, we have
\[ \|f|D|^sg + g|D|^sf-|D|^s(fg)\|_{L^p} \lesssim \||D|^{s_1}f\|_{L^{p_1}}\||D|^{s_2}g\|_{L^{p_2}}, \]
provided that $0<s<1$, $s=s_1+s_2$ and $s_1$, $s_2\geq 0$, and on the other side, $\frac{1}{p}=\frac{1}{p_1}+\frac{1}{p_2}$, with $p$, $p_1$, $p_2 \in (1,+\infty)$.
\end{lemme}
With this lemma, we can write
\begin{align} 2\Re e (\partial_x^{n}|D|u_1, u_2\partial_x^{n}&\dot{\overline{u_1}})= 2\Re e (|D|^{1/2}(\partial_x^{n}u_1), (|D|^{1/2}u_2)\partial_x^{n}\dot{\overline{u_1}}) \label{I}\\
&+2\Re e (|D|^{1/2}(\partial_x^{n}u_1), u_2|D|^{1/2}(\partial_x^{n}\dot{\overline{u_1}})) \label{II} \\ 
&+2\Re e (|D|^{1/2}(\partial_x^{n}u_1), |D|^{1/2}\left[u_2\partial_x^{n}\dot{\overline{u_1}}\right] -(|D|^{1/2}u_2)\partial_x^{n}\dot{\overline{u_1}} -u_2|D|^{1/2}(\partial_x^{n}\dot{\overline{u_1}})) .\label{III}
\end{align}
We estimate separately
\[ \begin{aligned} |\eqref{I}| &\lesssim \||D|^{1/2}\partial_x^nu_1\|_{L^4}\||D|^{1/2}u_2\|_{L^4}\|\partial_x^n\dot{\overline{u_1}}\|_{L^2}\\
&\lesssim \|u_1\|_{H^{\frac{3}{4}+n}}\|u_2\|_{H^\frac{3}{4}}\|u_1\|_{H^{1+n}} \\
&\lesssim (\|u_1^0\|_{H^{1/2}}^2+\|u_2^0\|_{H^{1/2}}^2)F(t).
\end{aligned}\]
On the other hand, \eqref{II} can be rewritten as $2\Re e (|D|^{1/2}(\partial_x^{n}\dot{u_1})|D|^{1/2}(\partial_x^{n}u_1), u_2)$, \textit{i.e.}
\[ \frac{d}{dt}\left[ \Re e \left( (\partial_x^{n}|D|^{1/2}u_1)^2, u_2 \right) \right] - \Re e \left( (\partial_x^{n}|D|^{1/2}u_1)^2, \dot{u_2}\right) .\]
Thus, perturbing the initial quantity by a term of lower order than $F(t)$, it is enough to control
\[ \left| \Re e \left( (\partial_x^{n}|D|^{1/2}u_1)^2, |D|u_2\right) \right| \leq \|u_1\|^2_{H^{\frac{3}{4}+n}}\|u_2\|_{H^1} \lesssim (\|u_1^0\|^2_{H^{1/2}}+\|u_2^0\|^2_{H^{1/2}})F(t).\]
Eventually, using $L^4$-$L^\frac{4}{3}$ duality, and lemma \ref{KPV} with $s=\frac{1}{2}$, $s_1=\frac{1}{2}$, $s_2=0$, and $p=\frac{4}{3}$, $p_1=4$, $p_2=2$, we infer that
\[ |\eqref{III}| \lesssim \||D|^{1/2}\partial_x^nu_1\|_{L^4}\||D|^{1/2}u_2\|_{L^4}\|\partial_x^n\dot{\overline{u_1}}\|_{L^2}\lesssim (\|u_1^0\|_{H^{1/2}}^2+\|u_2^0\|^2_{H^{1/2}}) F(t), \]
as above.

To sum up, there exist a constant $C_n$ depending only on $n$, such that, for all times $t\in \mathbb{R}$, $F(t) \leq 4F(0) +C_n(\|u_1^0\|^2_{H^{1/2}}+\|u_2^0\|^2_{H^{1/2}})\int_0^t F(s)ds$, which means, by Gronwall's lemma, that
\[ F(t) \leq 4F(0)e^{C_n(\|u_1^0\|_{H^{1/2}}^2+\|u_2^0\|_{H^{1/2}}^2)|t|}.\]
This tells us that $\|u_1\|_{H^{1+n}}$ and $\|u_2\|_{H^{1+n}}$ grow at most exponentially, and theorem \ref{theo^2} is proved.

\section{Dispersion estimates and Bourgain spaces}
We come back to equation \eqref{mmt} and to the end of the proof of theorem \ref{theommt}. We have to deal now with the case when $\alpha <1$. In this case, the boundedness of the $H^{\alpha /2}$-norm of the solutions is not enough to get a pointwise control of their $L^\infty$-norm. In other terms, we need to prove a Strichartz estimate for solutions of \eqref{mmt}.

\subsection{The Strichartz estimate}
From now on, $\alpha$ is fixed, with $\frac{2}{3}<\alpha <1$. For $u_0 \in \mathscr{D}'(\mathbb{T})$ and $t\in \mathbb{R}$, denote by
\[ S(t)u_0:=e^{-it|D|^\alpha}u_0= \sum_{k=-\infty}^{+\infty} \widehat{u_0}(k)e^{i(kx-|k|^\alpha t)},\]
the solution of the homogeneous equation $i\partial_t u= |D|^\alpha u$, with value $u_0$ at time $t=0$.

We are also going to use the Littlewood-Paley decomposition. For this purpose, let $\psi$ be a nonnegative $\mathcal{C}^\infty$ function on $\mathbb{R}_+$, such that $\psi >0$ on $\left[ \frac{1}{2}+\frac{1}{10},2-\frac{1}{10}\right]$ and $\psi \equiv 0$ outside $\left] \frac{1}{2},2\right[$. Without loss of generality, we can assume that $\sum_{j=1}^{+\infty} \psi(2^{-j}x) \equiv 1$ on $[2, +\infty)$. Let then $u\in \mathscr{D}'(\mathbb{T})$, and $N=2^j$ for some integer $j\geq 1$. We define
\[ \Delta_Nu:=\psi \left(\frac{|D|}{N}\right) u =\sum_{|k|\in \left[\frac{N}{2},2N \right]}\psi \left(\frac{|k|}{N}\right)\hat{u}(k)e^{ikx},\]
and $\Delta_1 u:= u-\sum_{j\geq 1} \Delta_{2^j}u$. In the sequel, capital letters will always refer to dyadic integers, and we will use the simplified notation $\sum_N$ for a sum over all dyadic integers, starting from $1$.

Let us recall a few facts about the Littlewood-Paley decomposition :
\begin{align}
u&= \sum_N \Delta_Nu ,\\
\|u\|_{H^s}^2 &\simeq \sum_N N^{2s}\|\Delta_Nu\|_{L^2}^2, \label{LPsob}\\
\|u\|_{L^p}^2 &\simeq \sum_N \|\Delta_Nu\|_{L^p}^2. \label{LPLp}
\end{align}
By the $\simeq$ sign, we just mean that the two quantities, as norms, are equivalent.

We can now state our Strichartz lemma :
\begin{lemme}[Strichartz inequality]\label{strich}
There exist a constant $C_\alpha>0$, depending on $\alpha$, such that, for every $u\in L^2(\mathbb{T})$ and every $N=2^j$,
\begin{equation}
\label{Sloc}
\|e^{-it|D|^\alpha}\Delta_Nu\|_{L^4\left( (0,1)_t, L^\infty({\mathbb{T}})\right)} \leq C_\alpha\|u\|_{L^2}N^{\frac{1}{2}-\frac{\alpha}{4}}.
\end{equation}
\end{lemme}
This lemma also has a non-localized version :
\begin{cor}\label{Sglob}
For every $\gamma > \frac{1}{2}-\frac{\alpha}{4}$ and every $u\in H^\gamma(\mathbb{T})$, we have
\[ \|e^{-it|D|^\alpha}u\|_{L^4\left( (0,1)_t, L^\infty({\mathbb{T}})\right)} \leq C_{\alpha,\gamma}\|u\|_{H^\gamma}.\]
\end{cor}

\begin{proof}[Proof of the corollary.]
Using the triangle inequality, \eqref{Sloc}, and the Cauchy-Schwarz inequality, we can write that
\begin{multline*} \|e^{-it|D|^\alpha}u\|_{L^4L^\infty} = \left\| \sum_N e^{-it|D|^\alpha}\Delta_Nu\right\|_{L^4L^\infty}\leq \sum_N\|e^{-it|D|^\alpha}\Delta_Nu\|_{L^4L^\infty} \leq C\sum_N N^{\frac{1}{2}-\frac{\alpha}{4}}\|\Delta_Nu\|_{L^2}\\
\leq C \left( \sum_N N^{2\gamma}\|\Delta_Nu\|^2_{L^2} \right)^{\frac{1}{2}}\left( \sum_N N^{-2\left(\gamma - (\frac{1}{2}-\frac{\alpha}{4})\right)} \right)^\frac{1}{2} \leq C_{\alpha ,\gamma}\|u\|_{H^\gamma},
\end{multline*}
where the last bound comes from \eqref{LPsob}.
\end{proof}

\begin{proof}[Proof of lemma \ref{strich}.]
To proove the Strichartz inequality, we proceed in a quite usual manner : we begin by showing a dispersion estimate, and to conclude, we apply a $TT^*$-argument, combined with the Hardy-Littlewood-Sobolev inequality.

Let $N=2^j$, $j\geq 1$, be a dyadic integer\footnote{For $N=1$, \eqref{Sloc} holds trivially.}. For $t\in \mathbb{R}$, we have
\begin{align*}
e^{-it|D|^\alpha}\Delta_Nu(x)&=\sum_{k\in \mathbb{Z}} \widehat{\Delta_Nu}(k)e^{i(kx-|k|^\alpha t)} \\&= \sum_{k\in \mathbb{Z}} \int_\mathbb{T}\psi\left( \frac{|k|}{N}\right) u(y)e^{i(k(x-y)-|k|^\alpha t)}dy\\&=:(u\ast_x \kappa_N)(x,t),
\end{align*} 
where $\kappa_N$ stands for the following kernel :
\[\kappa_N(x,t):=\sum_{k\in\mathbb{Z}}\psi\left( \frac{|k|}{N}\right) e^{i(kx-|k|^\alpha t)}.\]

Our first step will be to estimate $\|\kappa_N(\cdot,t)\|_{L^\infty(\mathbb{T})}$ for fixed $t\in (-1,1)$, $t\neq 0$. Applying the Poisson summation formula to the function $F_{x,t}(y):=\psi(|y|/N)e^{i(yx-|y|^\alpha t)}$, which is $\mathcal{C}^\infty$ and compactly supported, we have
\[\kappa_N(x,t)=\sum_{k\in \mathbb{Z}} F_{x,t}(k)=\sum_{n\in \mathbb{Z}} \widehat{F_{x,t}}(2\pi n) = \sum_{n\in \mathbb{Z}} \int_\mathbb{R} N\psi(|\xi |)e^{i\left[ N\xi(x-2\pi n) -N^\alpha t|\xi |^\alpha \right]}d\xi ,\]
and to study $\|\kappa_N\|_{L^\infty}$, we naturally restrict ourselves to $x\in (-\pi ,\pi]$.

The integrals above will be estimated by a stationnary phase result called the Van der Corput lemma (see \cite{Stein}) :

\begin{lemme}[Van der Corput] \label{VdC}
Let $\varphi$, $\Psi:\mathbb{R}\to \mathbb{C}$ be two smooth functions, with $\Psi$ compactly supported on $\mathbb{R}$. Suppose in addition that there exist $A>0$ such that $|\varphi ''|\geq A$ on $\supp (\Psi )$. Then
\[ \left| \int_\mathbb{R} e^{i\varphi(x)}\Psi(x)dx \right| \leq \frac{C}{\sqrt{A}}\int_\mathbb{R}|\Psi '(x)|dx,\]
where $C>0$ is an absolute constant.
\end{lemme}

In our case, the phase reads $ N\xi(x-2\pi n) -N^\alpha t |\xi|^\alpha$, and we denote it by $\phi_n(\xi)$. Compute $\phi_n'(\xi)=N(x-2\pi n)- \alpha N^\alpha t \xi|\xi|^{\alpha -2}$. In particular, $\phi_n'(\xi)=0$ if and only if
\begin{equation} \label{phi}\sgn (\xi)\left( N|\xi|\right) ^{1-\alpha}= \frac{\alpha t}{x-2\pi n}.
\end{equation}
Because of $\psi$ cutting all frequences below $\frac{1}{2}$, and because of the condition $1-\alpha >0$, we have $(N|\xi|)^{1-\alpha}\geq 1$ on $\supp (\psi)$, whereas $|t\alpha |\leq 1$. Furthermore, when $n\neq 0$, $|x-2\pi n|^{-1} \leq \pi ^{-1}$ : in that case, \eqref{phi} cannot hold.

So suppose first $n\neq 0$. Then $\|\phi_n'\|_{L^\infty(\supp (\psi))}\geq N|x-2\pi n|-N^\alpha 2^{1-\alpha}$, and integrating by parts, it is easy to estimate the integral
\[ I_n:=\int_\mathbb{R} N\psi(|\xi |)e^{i\phi_n(\xi )}d\xi =\int_{\mathbb{R}} N\frac{d}{d\xi}\left( \frac{1}{i\phi_n'(\xi)}\frac{d}{d\xi} \left(\frac{\psi(|\xi |)}{i\phi_n'(\xi )}\right) \right) e^{i\phi_n(\xi )}d\xi. \]
Thus, we have $|I_n|\leq CN\|\phi_n'\|_{L^\infty}^{-2}\|\psi\|_{H^2}$, where $C$ is proportionnal to the size of $\supp (\psi)$. Finally, we sum on $n$ :
\begin{equation} \label{disp1}\left| \sum_{n\neq 0} I_n \right| \leq \frac{C}{N}\sum_{n\neq 0}\left( |x-2\pi n|-\left( \frac{2}{N} \right) ^{1-\alpha} \right) ^{-2}\leq \frac{C}{N}\sum_{n\neq 0} (2\pi |n|-\pi-1)^{-2}\leq \frac{\tilde{C}}{N},
\end{equation}
with $\tilde{C}$ just depending on $\psi$.

The only difficult part, then, is $I_0$, because \eqref{phi} could be satisfied. At this point, we apply lemma \ref{VdC}, and calculate $\phi_0''(\xi )=\alpha (1-\alpha)N^\alpha t |\xi |^{\alpha -2}$. It is clear that $|\phi_0''(\xi )|\geq \alpha (1-\alpha ) N^\alpha |t| 2^{2-\alpha}$ on $\supp (\psi)$, so we have
\begin{equation}\label{disp2} |I_0| \leq C\frac{N^{1-\frac{\alpha}{2}}}{\sqrt{|t|}},
\end{equation}
where $C>0$ is a constant depending on $\alpha$ and $\psi$.

So far, \eqref{disp1} and \eqref{disp2} show that there exist a constant $C>0$, depending only on $\alpha$ and $\psi$, such that $\| \kappa_N(\cdot, t) \|_{L^\infty} \leq C|t|^{-1/2}N^{1-\alpha/2}$ for all $t\in (-1,1)$, $t\neq 0$. In particular, for fixed $t$, considering $S(t)\Delta_N$ as an operator mapping $L^1(\mathbb{T})$ to $L^\infty (\mathbb{T})$, we have
\begin{equation}\label{disp3} \|S(t)\Delta_N\|_{L^1\to L^\infty} \leq C\frac{N^{1-\frac{\alpha}{2}}}{\sqrt{|t|}}.
\end{equation}

\vspace{1em}
Now comes the $TT^*$-argument. Define a linear operator $T: u\mapsto S(t)\Delta_Nu$. We want to prove that $T$ maps $L^2(\mathbb{T})$ into $L^4((0,1)_t,L^\infty(\mathbb{T}))$, as well as to bound its norm. To this end, we rather study the operator $TT^*$, where $T^*: L^{4/3}((0,1)_t,L^1(\mathbb{T}))\to L^2(\mathbb{T})$ is (a restriction of) the adjoint of $T$. We can find $T^*$ explicitely. Let $g\in L^{4/3}L^1$ : for $u\in L^2(\mathbb{T})$,
\[ \iint_{(0,1)\times\mathbb{T}}[S(t)\Delta_Nu(x)]\overline{g(t,x)}dxdt=\left( u,\int_{0}^1 S(-s)\Delta_Ng(s,x)\right)_{L^2(\mathbb{T})}.\]
Thus, $TT^*(g)(t,x)=\int_{0}^1 \Delta_NS(t-s)\Delta_Ng(s,x)ds$, and by \eqref{disp3}, for all $t\in (0,1)$,
\[ \|TT^*(g)(t,\cdot)\|_{L^\infty(\mathbb{T})} \leq C \int_0^1 \frac{N^{1-\frac{\alpha}{2}}}{\sqrt{|t-s|}}\|g(s,\cdot )\|_{L^1(\mathbb{T})}ds.\]
In the integral of the left hand side, we recognize a convolution product between $t\mapsto \|g(t,\cdot )\|_{L^1(\mathbb{T})}$ and the function $\omega : t\mapsto |t|^{-1/2}$. The Hardy-Littlewood-Sobolev inequalities guarantee that the convolution with $\omega$ maps $L^{4/3}((0,1)_t)$ to $L^4((0,1)_t)$. In other terms, $\| TT^*(g)\|_{L^4L^\infty}\leq CN^{1-\alpha /2}\|g\|_{L^{4/3}L^1}$, which implies that the operator norm of $T$ is bounded :
\[ \|T\|_{L^2\to L^4((0,1)_t,L^\infty(\mathbb{T}))} \leq CN^{\frac{1}{2}-\frac{\alpha}{4}}.\]
This finishes the proof of lemma \ref{strich}.
\end{proof}

\begin{rem}
Notice that the results of lemma \ref{Sloc} and corollary \ref{Sglob} remain true, with the same constants, when replacing $(0,1)_t$ by any time interval of length $1$. This follows from the fact that $S(t)$ is an isometry in any $H^s(\mathbb{T})$, $s\geq 0$.
\end{rem}

\subsection{Bourgain spaces and embedding results}
The Strichartz estimate of corollary \ref{Sglob} will enable us to prove the local well-posedness of equation \eqref{mmt} in a certain Hilbert space, usually called a Bourgain space, which we are now going to define.

\begin{déf}
Let $u : \mathbb{R}\times\mathbb{T} \to\mathbb{C}$, and $s$, $b \in \mathbb{R}$.
\begin{itemize}
\item We say that $u \in H^{s,b}$ if for all $t\in \mathbb{R}$, $u(t)\in H^s(\mathbb{T})$, and if in addition, the function $t\mapsto \|u(t)\|_{H^s(\mathbb{T})}$ belongs to $H^b(\mathbb{R})$. Then, the norm $\|u\|_{H^{s,b}}$ is just the $H^b$-norm of $t\mapsto \|u(t)\|_{H^s(\mathbb{T})}$.
\item We say that $u \in X^{s,b}_\alpha$ if the function $v:(t,x)\mapsto S(-t)u(t,x)$ belongs to $H^{s,b}$. Then, we define $\|u\|_{X^{s,b}_\alpha}:=\|v\|_{H^{s,b}}=\|S(-t)u(t,x)\|_{H^{s,b}}$.
\end{itemize}
\end{déf}

The space $X_\alpha^{s,b}$ is called a \emph{Bourgain space}. Explicitely,
\begin{equation}\label{Bnorm}\|u\|_{X^{s,b}_\alpha}^2 =\frac{1}{2\pi} \sum_{k\in \mathbb{Z}}\int_\mathbb{R}(1+|k|^2)^s(1+|\tau +|k|^\alpha|^2)^b|\mathcal{F}u(\tau, k)|^2d\tau ,
\end{equation}
where $\mathcal{F}u(\cdot ,k)$ stands for the Fourier transform of $\widehat{u(\cdot ,k)}$ with respect to time, \textit{i.e.} the Fourier transform in both time- and space-variables.

Bourgain spaces are very convenient for several reasons. Playing on the two exponents $s$ and $b$, we begin by showing two embedding results :

\begin{lemme}\label{l2}
For any $b>\frac{1}{4}$, we have $\|u\|_{L^4(\mathbb{R}_t,L^2(\mathbb{T}))} \lesssim \|u\|_{X^{0,b}_\alpha}$.
\end{lemme}

\begin{proof}[Proof]
Assuming $u\in X^{0,2b}_\alpha$, write $\mathcal{F}u(\tau ,k)=\int_\mathbb{R}\widehat{u(t,k)}e^{-it\tau}dt$, so by the inverse Fourier transform and the Cauchy-Schwarz inequality (observing that $4b>1$),
\begin{multline*}
|\widehat{u(t,k)}|^2=\left| \frac{1}{2\pi}\int_\mathbb{R}\mathcal{F}u(\tau ,k)e^{it\tau}d\tau \right|^2 \\ \leq C \left( \int_\mathbb{R} |\mathcal{F}u(\tau ,k)|^2(1+|\tau +|k|^\alpha|^2)^{2b}d\tau \right) \left( \int_\mathbb{R} \frac{d\tau}{(1+|\tau +|k|^\alpha|^2)^{2b}} \right).
\end{multline*}
Summing over $k\in \mathbb{Z}$, we find $\|u(t)\|_{L^2(\mathbb{T})}^2 \leq C_b\|u\|_{X^{0,2b}_\alpha}^2$, or equivalently $\|u\|_{L^\infty(\mathbb{R}_t,L^2(\mathbb{T})) }\leq C\|u\|_{X^{0,2b}_\alpha}$. But the equality $\|u\|_{L^2(\mathbb{R}_t,L^2(\mathbb{T}))} = \|u\|_{X^{0,0}_\alpha}$ also follows from \eqref{Bnorm} and the Parseval formula. Interpolating between these two statements gives the result.
\end{proof}

The following lemma is a consequence of the Strichartz inequality.
\begin{lemme}\label{linfty}
For any $b>\frac{1}{2}$ and $\gamma >\frac{1}{2}-\frac{\alpha}{4}$, we have $\|u\|_{L^4(\mathbb{R}_t,L^\infty(\mathbb{T}))}\lesssim \|u\|_{X^{\gamma ,b}_\alpha}$.
\end{lemme}

\begin{proof}[Proof]
Let $u\in X^{\gamma ,b}_\alpha$, and $v:=S(-t)u$. Suppose at first that $t \mapsto u(t,\cdot )$ is supported on an interval $I_t$ of length $1$. Thus it is possible to apply corollary \ref{Sglob} directly : indeed,
\begin{multline*}
\|u\|_{L^4(\mathbb{R}_t,L^\infty)}=\|S(t)v\|_{L^4(I_t,L^\infty)}=\left\| S(t) \int_\mathbb{R}\hat{v}(\tau )e^{it\tau }d\tau \right\|_{L^4(I_t,L^\infty)} \leq \int_\mathbb{R} \|S(t)\hat{v}(\tau)\|_{L^4(I_t,L^\infty)}d\tau \\
\leq C\int_\mathbb{R}\|\hat{v}(\tau )\|_{H^\gamma}d\tau \leq C \left( \int_\mathbb{R} \|\hat{v}(\tau)\|_{H^\gamma}^2(1+|\tau |^2)^bd\tau \right)^{\frac{1}{2}} \left( \int_\mathbb{R} \frac{d\tau}{(1+|\tau |^2)^b} \right)^\frac{1}{2} \leq C_b\|v\|_{H^{\gamma,b}}.
\end{multline*}
Since $\|v\|_{H^{\gamma ,b}}=\|u\|_{X^{\gamma ,b}_\alpha}$, this finishes the first part of the proof.

Now, we remove the special assumption on $u$. By simple construction, it is possible to find a function $\vartheta \in \mathcal{C}^\infty_0((0,1))$, such that $0\leq \vartheta\leq 1$ on $\mathbb{R}$, and $\sum_{n\in \mathbb{Z}} \vartheta (t-n/2) = 1$, for all $t\in\mathbb{R}$. We have
\begin{align*} \|u&\|^4_{L^4(\mathbb{R}_t, L^\infty (\mathbb{T}))}= \int_{\mathbb{R}} dt\left\| \sum_{n\in\mathbb{Z}} u(t,\cdot)\vartheta \left( t-\frac{n}{2}\right) \right\|_{L^\infty (\mathbb{T})}^4 \leq \int_\mathbb{R}dt\left| \sum_{n\in\mathbb{Z}}\vartheta \left( t-\frac{n}{2}\right) \|u(t,\cdot)\|_{L^\infty(\mathbb{T})}\right| ^4 \\
&= \int_\mathbb{R}dt \sum_{\substack{n_1\in\mathbb{Z}\\n_2,n_3,n_4\in\lbrace n_1-1,n_1,n_1+1\rbrace}}\vartheta \left( t-\frac{n_1}{2}\right) \vartheta \left( t-\frac{n_2}{2}\right) \vartheta \left( t-\frac{n_3}{2}\right) \vartheta \left( t-\frac{n_4}{2}\right) \|u(t,\cdot)\|^4_{L^\infty(\mathbb{T})} \\
&\leq C\sum_{n_1\in\mathbb{Z}}\left\| \vartheta \left( t-\frac{n_1}{2} \right)u(t,\cdot)\right\|_{L^4(\mathbb{R}_t,L^\infty(\mathbb{T}))}^4 ,
\end{align*}
thanks to the elementary inequality : $abcd\leq \frac{1}{4}(a^4+b^4+c^4+d^4)$. To each term of the sum, we apply the first part of the proof, and we find, using the embedding $\ell^4(\mathbb{N})\hookrightarrow \ell^2(\mathbb{N})$ :
\begin{align*}
\|u\|_{L^4(\mathbb{R}_t, L^\infty (\mathbb{T}))} &\lesssim \left( \sum_{n\in\mathbb{Z}} \left\| \vartheta \left( t-\frac{n}{2} \right) S(-t)u(t,\cdot )\right\|_{H^{\gamma, b}}^4 \right) ^\frac{1}{4} \\
&\lesssim \left( \sum_{n\in\mathbb{Z}} \left\| \vartheta \left( t-\frac{n}{2} \right) S(-t)u(t,\cdot )\right\|_{H^{\gamma, b}}^2 \right) ^\frac{1}{2} \\
&\lesssim \|S(-t)u(t,\cdot )\|_{H^{\gamma ,b}}.
\end{align*}
The very last bound comes from the following classical lemma :
\begin{lemme}
For any $ b\in [0,1] $, any function $w\in H^b(\mathbb{R})$, any smooth $\vartheta $ as above, the following norms are equivalent :
\[ \|w\|_{H^b(\mathbb{R})}\simeq \left( \sum_{n\in \mathbb{Z}} \|\vartheta (\cdot -n/2)w(\cdot )\|_{H^b(\mathbb{R})}^2 \right) ^\frac{1}{2}.\]
\end{lemme}
\end{proof}

\subsection{The nonlinear estimate}
To solve \eqref{mmt} locally in time, we need to introduce a restriction space. For $T>0$, let $X^{s,b}_\alpha(T)$ be the set of all functions $u$ defined on $[-T,T]$ such that there exist a function $\tilde{u}\in X^{s,b}_\alpha$ with $\tilde{u}|_{[-T,T]}\equiv u$. Endowed with the norm
\[ \|u\|_{X^{s,b}_\alpha(T)} := \inf \left\lbrace \|\tilde{u}\|_{X^{s,b}_\alpha} \middle| \:\tilde{u}\in X^{s,b}_\alpha , \: \tilde{u}\equiv u \text{ on }[-T,T] \right\rbrace ,\]
$X^{s,b}_\alpha(T)$ is a Banach space, so we can apply Picard's fixed-point theorem in $X^{s,b}_\alpha(T)$. From here on, we strictly follow the scheme of the proof of \cite[Theorem 3]{BGT}.

Fix a function $\varphi \in C^\infty_0(\mathbb{R})$, such that $\varphi(t)=1$ when $|t|\leq 1$. Then for $T\in (0,1]$, a solution of \eqref{mmt} on the interval $[-T,T]$ is a fixed point of the application
\begin{equation}\label{fix}
\mathcal{K}: u \longmapsto \varphi(t)S(t)u_0-i\varphi \left( \textstyle\frac{t}{T}\displaystyle\right) \int_0^t S(t-s)\left[ |u(s)|^2u(s)\right] ds.
\end{equation}

Thanks to the definition of Bourgain spaces, estimating the first part of $\mathcal{K}$ is elementary : $\|\varphi(t)S(t)u_0\|_{X^{s,b}_\alpha}=\|\varphi\|_{H^b(\mathbb{R}_t)}\|u_0\|_{H^s(\mathbb{T})}$. All the difficulty lies in the nonlinear part, and we now turn to our central proposition.

\begin{prop}\label{nonl}Let $\gamma > \frac{1}{2}-\frac{\alpha}{4}$, $b>\frac{1}{2}$, $b'>\frac{1}{4}$ such that $b+b'<1$. Then
\[\|u_1\overline{u_2}u_3\|_{X^{\gamma,-b'}_\alpha} \leq C\|u_1\|_{X^{\gamma,b}_\alpha}\|u_2\|_{X^{\gamma,b}_\alpha}\|u_3\|_{X^{\gamma,b}_\alpha}.\]
\end{prop}

\begin{proof}[Proof]
We prove this proposition thanks to a duality argument. Since $(X^{\sigma, \beta}_\alpha)'=X^{-\sigma,-\beta}_\alpha$ for any $\sigma$, $\beta \in \mathbb{R}$, it is sufficent to show that, for a given $u_0 \in X^{-\gamma,b'}_\alpha$, we have
\begin{equation}\label{dual} \left| \int_{\mathbb{R}\times\mathbb{T}}u_1\overline{u_2}u_3\overline{u_0} \right| \leq \|u_0\|_{X^{-\gamma ,b'}_\alpha}\prod_{j=1}^3\|u_j\|_{X^{\gamma ,b}_\alpha}.
\end{equation}
In fact, we restrict the proof of \eqref{dual} to the case of smooth functions $u_j$, $u_0$, with compact support in time --- and then conclude by density of such functions. To start with, we introduce new functions $w_0$, $w_1$, $w_2$, $w_3$ by
\begin{align*}
\mathcal{F}w_j(\tau ,k)&:= (1+|k|^2)^\frac{\gamma}{2}(1+|\tau +|k|^\alpha|^2)^\frac{b}{2}\mathcal{F}u_j(\tau, k) ,\quad \text{for }j\in \lbrace 1,2,3 \rbrace ,\\
\mathcal{F}w_0(\tau ,k)&:= (1+|k|^2)^{-\frac{\gamma}{2}}(1+|\tau +|k|^\alpha|^2)^{\frac{b'}{2}}\mathcal{F}u_0(\tau, k),
\end{align*}
and we can replace the right hand side of \eqref{dual} by $\|w_0\|_{L^2(\mathbb{R}\times\mathbb{T})}\prod_{j=1}^3\|w_j\|_{L^2(\mathbb{R}\times\mathbb{T})}$.

To estalish \eqref{dual}, we perform a time- and space-localization. By $L_0$, $L_j$, $N_0$, $N_j$, we refer to dyadic integers, and $L:=(L_0,L_1,L_2,L_3)$ will concern time, whereas $N:=(N_0,N_1,N_2,N_3)$ will hint at space. Define, for $j\in \lbrace 1,2,3 \rbrace$,
\begin{align*}
u_j^{L_jN_j}(t,x)&:= \frac{1}{2\pi} \sum_{|n|\in [N_j,2N_j[}e^{inx}\int_{L_j\leq |\tau +|n|^\alpha|<2L_j}(1+|n|^2)^{-\frac{\gamma}{2}}(1+|\tau +|n|^\alpha|^2)^{-\frac{b}{2}}\mathcal{F}w_j(\tau, n)e^{it\tau}d\tau ,\\
u_0^{L_0N_0}(t,x)&:= \frac{1}{2\pi} \sum_{|n|\in [N_0,2N_0[}e^{inx}\int_{L_0\leq |\tau +|n|^\alpha|<2L_0}(1+|n|^2)^{\frac{\gamma}{2}}(1+|\tau +|n|^\alpha|^2)^{-\frac{b'}{2}}\mathcal{F}w_0(\tau, n)e^{it\tau}d\tau .
\end{align*}
A simple calculus shows that, for $j\in \lbrace 1,2,3 \rbrace$ first,
\[ \begin{aligned} \|u_j^{L_jN_j}\|_{X^{\sigma ,\beta}_\alpha}^2 &= C \sum_{|n|\simeq N_j}\int_{|\tau +|n|^\alpha |\simeq L_j}(1+|n|^2)^{\sigma -\gamma}(1+|\tau +|n|^\alpha|^2)^{\beta -b}|\mathcal{F}w_j(\tau, n)|^2d\tau \\
&\lesssim L_j^{2(\beta -b)}N_j^{2(\sigma-\gamma)}\underbrace{\sum_{|n|\simeq N_j}\int_{|\tau +|n|^\alpha |\simeq L_j}|\mathcal{F}w_j(\tau, n)|^2d\tau}_{=: \: c_j(L_j,N_j)^2},
\end{aligned}\]
where $c_j(L_j,N_j)$ satisfies $\sum_{L_j} \sum_{N_j} c_j(L_j,N_j)^2 \leq \|w_j\|_{L^2(\mathbb{R}\times\mathbb{T})}^2$. (Here, and in all the sequel, the summation over $L_j$ or $N_j$ means the summation over all dyadic integers.) We have a similar result for $u_0^{L_0N_0}$, so finally, for any $\sigma$, $\beta \in \mathbb{R}$,
\begin{align}
\|u_j^{L_jN_j}\|_{X^{\sigma ,\beta}_\alpha} &\lesssim L_j^{\beta -b}N_j^{\sigma-\gamma}c_j(L_j,N_j), \label{uj}\\
\|u_0^{L_0N_0}\|_{X^{\sigma ,\beta}_\alpha} &\lesssim L_0^{\beta -b'}N_0^{\sigma+\gamma}c_0(L_0,N_0). \label{u0}
\end{align}

Now we are going to estimate
\[ I(L,N):=\left| \int_{\mathbb{R}\times\mathbb{T}} u_1^{L_1N_1}\overline{u_2^{L_2N_2}}u_3^{L_3N_3}\overline{u_0^{L_0N_0}} \right| .\]
Notice that integrating on $\mathbb{T}$ implies that $I(L,N)=0$ unless $N_0\leq 2(N_1+N_2+N_3)$. From now on, we suppose that this condition is fulfilled. Moreover, the proof below does not take into account the precise role of the $u_j$'s, neither the conjugate bar, so we can assume that $N_1\geq N_2 \geq N_3$.

Using the Hölder inequalities, lemmas \ref{l2} and \ref{linfty}, and finally \eqref{uj} and \eqref{u0}, choosing any $\beta \in (\frac{1}{2},b)$, $\beta '\in (\frac{1}{4},b')$, $\gamma ' \in ( \frac{1}{2}-\frac{\alpha}{4},\gamma)$, we bound $I(L,N)$ :
\[ \begin{aligned} I(L,N) &\leq \|u_1^{L_1N_1}\|_{L^4(\mathbb{R}_t,L^2(\mathbb{T}))}\cdot \|u_2^{L_2N_2}\|_{L^4(\mathbb{R}_t,L^\infty(\mathbb{T}))}\cdot \|u_3^{L_3N_3}\|_{L^4(\mathbb{R}_t,L^\infty(\mathbb{T}))}\cdot \|u_0^{L_0N_0}\|_{L^4(\mathbb{R}_t,L^2(\mathbb{T}))} \\
&\leq \|u_1^{L_1N_1}\|_{X^{0,\beta '}_\alpha}\cdot \|u_2^{L_2N_2}\|_{X^{\gamma ' ,\beta}_\alpha}\cdot \|u_3^{L_3N_3}\|_{X^{\gamma' ,\beta}_\alpha}\cdot \|u_0^{L_0N_0}\|_{X^{0,\beta '}_\alpha} \\
&\leq \frac{(L_0L_1)^{\beta '}(L_2L_3)^\beta}{L_0^{b'}(L_1L_2L_3)^b}(N_2N_3)^{\gamma'}\frac{N_0^\gamma}{(N_1N_2N_3)^\gamma} \prod_{j=0}^3 c_j(L_j,N_j) \\
&\leq L_0^{-\varepsilon_0}L_1^{-\varepsilon_1}(L_2L_3)^{-\varepsilon '}(N_2N_3)^{-\eta '}\left( \frac{N_0}{N_1}\right) ^{\gamma} \prod_{j=0}^3 c_j(L_j,N_j),
\end{aligned}\]
where $\varepsilon_0$, $\varepsilon_1$, $\varepsilon '$, $\eta '$ are some positive constants. Consequently, we can sum on $L_2, N_2, L_3, N_3$, making use of the bound $c_j(L_j,N_j)\leq \|w_j\|_{L^2(\mathbb{R}\times\mathbb{T})}$ for $j\in \lbrace 2,3 \rbrace$. Then, by Cauchy-Schwarz,
\[ \begin{aligned} \sum_{L_0,L_1} \sum_{L_2, N_2, L_3, N_3}I(L,N)&\lesssim \left( \frac{N_0}{N_1}\right) ^{\gamma}\|w_2\|_{L^2}\|w_3\|_{L^2} \left( \sum_{L_0}c_0(L_0,N_0)L_0^{-\varepsilon_0} \right) \left(\sum_{L_1}c_1(L_1,N_1)L_1^{-\varepsilon_1} \right)  \\
&\lesssim \left( \frac{N_0}{N_1}\right) ^{\gamma}K_0(N_0)^\frac{1}{2}K_1(N_1)^\frac{1}{2}\|w_2\|_{L^2}\|w_3\|_{L^2},
\end{aligned}\]
introducing $K_j(N_j):=\sum_{L_j}c_j(L_j,N_j)^2$ for $j\in \lbrace 0,1 \rbrace$. It remains to sum on $N_0$, $N_1$, remembering that $N_0\leq 6N_1$, and using Cauchy-Schwarz again :
\[ \begin{aligned} \sum_{L,N} I(L,N) &\lesssim \sum_{\ell = -3}^{+\infty}\sum_{N_0} \left( \frac{N_0}{2^\ell N_0}\right) ^{\gamma}K_0(N_0)^\frac{1}{2}K_1(2^\ell N_0)^\frac{1}{2}\|w_2\|_{L^2}\|w_3\|_{L^2} \\
&\lesssim \sum_{\ell = -3}^{+\infty}2^{-\gamma\ell}\left( \sum_{N_0} K_0(N_0) \right)^\frac{1}{2}\left( \sum_{N_0}K_1(2^\ell N_0) \right)^\frac{1}{2}\|w_2\|_{L^2}\|w_3\|_{L^2} \\
&\lesssim \|w_0\|_{L^2}\|w_1\|_{L^2}\|w_2\|_{L^2}\|w_3\|_{L^2}.
\end{aligned}\]
Hence \eqref{dual} is proven, and so is proposition \ref{nonl}.
\end{proof}

\begin{rem}The condition $b+b'<1$ has not been used in the proof, except for the fact that $b'<b$ ; but it will be crucial in the next proposition.
\end{rem}

\subsection{Local and global well-posedness}
We are now ready to state our

\begin{prop}
Let $\gamma >\frac{1}{2}-\frac{\alpha}{4}$ and $\frac{1}{2}<b<1$. If $u_0\in H^\gamma(\mathbb{T})$, there exist $T_0>0$, depending only on $\|u_0\|_{H^\gamma}$, such that the problem \eqref{mmt} admits a unique solution $u\in X^{\gamma ,b}_\alpha(T)$ for all $T\leq T_0$. This solution satisfies $\|u\|_{X^{\gamma ,b}_\alpha(T)}\leq C\|u_0\|_{H^\gamma(\mathbb{T})}$, where $C>0$ is an absolute constant.

If in addition $u_0$ belongs to $H^s(\mathbb{T})$ for some $s>\gamma$, then $u(t)\in H^s(\mathbb{T})$ for all $t\in [-T_0,T_0]$.
\end{prop}

\begin{proof}[Proof]
We intend to show that the functional $\mathcal{K}$, defined in \eqref{fix}, is a contraction in some ball of the space $X^{\gamma ,b}_\alpha(T)$, for well-chosen $T$.

Since the first part of $\mathcal{K}$ has been previously bounded, we turn to
\[ M_2(u):=\left\| \varphi \left( \textstyle\frac{t}{T}\displaystyle \right) \int_0^t S(t-s)\left[ |u(s)|^2u(s)\right] ds\right\|_{X^{\gamma ,b}_\alpha(T)} = \left\| \varphi \left( \textstyle\frac{t}{T}\displaystyle \right) \int_0^t S(-s)\left[ |u(s)|^2u(s)\right] ds\right\| _{H^{\gamma ,b}}.\]
Here, we take advantage of the regularizing property of time-integration. Let $U\in \mathcal{C}^\infty_0(\mathbb{R}\times\mathbb{T})$, and $b$, $b'$ as in proposition \ref{nonl}. Set $G(t,x):= \varphi \left( \frac{t}{T}\right) \int_0^t U(s,x) ds$. For fixed $k\in \mathbb{Z}$, a lemma of Ginibre \cite[lemma (3.11)]{gin} guarantees that
\[\|\hat{G}(t,k)\|_{H^b(\mathbb{R}_t)}=\left\| \varphi \left( \textstyle\frac{t}{T}\displaystyle\right) \int_0^t \hat{U}(s,k) ds\right\|_{H^b(\mathbb{R}_t)} \lesssim T^{1-b-b'}\|\hat{U}(t,k)\|_{H^{-b'}(\mathbb{R}_t)},\]
with an implicit constant which only depends on $\varphi$, $b$, $b'$. Squaring this identity, we find
\[ \int_{\mathbb{R}} (1+|\tau |^2)^b|\mathcal{F}(G)(\tau ,k)|^2 d\tau \lesssim T^{1-b-b'}\int_{\mathbb{R}} (1+|\tau |^2)^{-b'}|\mathcal{F}(U)(\tau ,k)|^2 d\tau.\]
Eventually, multiply by $(1+|k|^2)^\gamma$ and sum over $k$ to get $\|G\|_{H^{\gamma ,b}}\lesssim T^{1-b-b'}\|U\|_{H^{\gamma ,-b'}}$, which remains true, by approximation, for less regular functions.

As a consequence,
\[M_2(u)\lesssim T^{1-b-b'}\|S(-t)\left[ |u(t)|^2u(t)\right]\|_{H^{\gamma ,-b'}}=T^{1-b-b'}\||u|^2u\|_{X^{\gamma ,-b'}_\alpha} \lesssim T^{1-b-b'}\|u\|_{X^{\gamma ,b}_\alpha}^3, \]
by proposition \ref{nonl}. This calculation is valid for any $\tilde{u}\in X^{\gamma ,b}_\alpha$ such that $\tilde{u}=u$ on $[-T,T]$. Thus, we proved that
\begin{equation}\label{kontract} \|\mathcal{K}(u)\|_{X^{\gamma ,b}_\alpha(T)} \leq \|\varphi\|_{H^b(\mathbb{R}_t)}\|u_0\|_{H^\gamma(\mathbb{T})} + CT^{1-b-b'}\|u\|_{X^{\gamma ,b}_\alpha(T)}^3.
\end{equation}
So $\mathcal{K}$ stabilizes the ball $B$ centered at the origin, of radius $\tilde{C}\|u_0\|_{H^\gamma}$ (for some arbitrary $\tilde{C}> \|\varphi\|_{H^b(\mathbb{R}_t)}$), provided that $T\leq T_0$ with
\[T_0:=\left(\frac{\tilde{C}-\|\varphi\|_{H^b(\mathbb{R}_t)}}{\tilde{C}^3C\|u_0\|_{H^\gamma}^2} \right)^\frac{1}{1-b-b'}.\]

Reasoning in the same way, by means of the identity
\[|u|^2u-|v|^2v=u(\overline{u-v})u+(u-v)\bar{v}(u+v),\]
we show that for $T\leq T_0$\footnote{Or possibly a fixed fraction of $T_0$.}, there exist a positive constant $c<1$ such that $\|\mathcal{K}(u)-\mathcal{K}(v)\|_{X^{\gamma ,b}_\alpha(T)} \leq c\|u-v\|_{X^{\gamma ,b}_\alpha(T)}$ for all $u$, $v\in B$. Hence $\mathcal{K}:B\to B$ is a contraction, and has a fixed point, also called $u$. Since $u\in B$, we have $\|u\|_{X^{\gamma ,b}_\alpha(T)}\leq \tilde{C}\|u_0\|_{H^\gamma(\mathbb{T})}$.

To prove the uniqueness of $u$, notice first, in view of \eqref{kontract}, that any other fixed point of $\mathcal{K}$ in $X^{\gamma ,b}_\alpha(T)$, as a function of some $X^{\gamma ,b}_\alpha(\tilde{T})$ for some smaller $\tilde{T}\leq T$, lies in the ball centered at $0$ and of radius $\tilde{C}\|u_0\|_{H^\gamma}$, so equals $u$ in that space, by unicity of the fixed point. Now let $v\in X^{\gamma ,b}_\alpha(T)$ be another solution of \eqref{mmt}. Observe that, because $b>\frac{1}{2}$, both $u$ and $v$ are continuous functions from $[-T,T]$ to $H^\gamma(\mathbb{T})$. Define $T_1:=\sup \lbrace t\in [0,T]\mid u(t)=v(t)\rbrace$, and suppose $T_1<T$. Then, translating time, and restarting the equation with $u(T_1)=v(T_1)$ as an initial data, we get a contradiction, by the previous remark.

Finally, if $u_0\in H^s(\mathbb{T})$ for $s>\gamma$, and if $u$ is the associated solution in $X^{\gamma ,b}_\alpha(T_0)$, let us show that $u(t)\in H^s(\mathbb{T})$ for all $|t|\leq T_0$. It is crucial to see, modifying slightly the proof of proposition \ref{nonl}, that whenever $s>\gamma>\frac{1}{2}-\frac{\alpha}{4}$, and $b$, $b'$ as above,
\[\|u_1\overline{u_2}u_3\|_{X^{s,-b'}_\alpha} \leq C\sum_{j=1}^3\left( \|u_j\|_{X^{s,b}_\alpha}\prod_{k\neq j}\|u_k\|_{X^{\gamma,b}_\alpha}\right) .\]
Given $\tilde{u}$ in the intersection of the ball of radius $\tilde{C}\|u_0\|_{H^s}$ in the space $X^{s,b}_\alpha(T)$, and of the ball of radius $C'\|u_0\|_{H^\gamma}$ in the space $X^{\gamma ,b}_\alpha(T_0)$, \eqref{kontract} becomes, for $T\leq T_0$ :
\begin{align*}
\|\mathcal{K}(\tilde{u})\|_{X^{s ,b}_\alpha(T)} &\leq \|\varphi\|_{H^b(\mathbb{R}_t)}\|u_0\|_{H^s(\mathbb{T})} + CT^{1-b-b'}\|\tilde{u}\|_{X^{\gamma ,b}_\alpha(T)}^2\|\tilde{u}\|_{X^{s ,b}_\alpha(T)} \\
&\leq \left( \|\varphi\|_{H^b(\mathbb{R}_t)}+C\tilde{C}C'^2T^{1-b-b'}\|u_0\|_{H^\gamma(\mathbb{T})}^2\right) \|u_0\|_{H^s(\mathbb{T})}.
\end{align*}
This shows that if $T$ is chosen small enough, \emph{regardless} of the size of $\|u_0\|_{H^s}$, $\mathcal{K}$ stabilizes the set we described above. The same can be done while estimating $\|\mathcal{K}(u)-\mathcal{K}(v)\|_{X^{s ,b}_\alpha(T)}$, and $\mathcal{K}$ therefore has a fixed point $\tilde{u}$ in some $X^{s,b}_\alpha(T)$. Obviously, since $X^{s ,b}_\alpha(T)\hookrightarrow X^{\gamma ,b}_\alpha(T)$, we have $\tilde{u}\equiv u$ on $[-T,T]$. Repeating this argument after translating time and restarting the equation from $\tilde{u}(T)=u(T)$, the claim is proved.
\end{proof}

The next corollary follows as an immediate consequence, and uses explicitly the condition $\alpha>\frac{2}{3}$ :

\begin{cor}\label{cor10}
Let $u_0 \in \mathcal{C}^\infty(\mathbb{T})$. Then \eqref{mmt} admits a unique global solution $u\in \mathcal{C}(\mathbb{R},\mathcal{C}^\infty(\mathbb{T}))$.

Besides, let $b\in (\frac{1}{2},1)$. For any $\gamma \in (\frac{1}{2}-\frac{\alpha}{4}, \frac{\alpha}{2}]$, there exist $T_0^{(\gamma)}$,  $C_\gamma >0$, such that for any $t\in \mathbb{R}$,
\begin{equation}\label{moy} \|u(\cdot -t)\|_{X^{\gamma ,b}_\alpha(T_0^{(\gamma)})}\leq C_\gamma\|u_0\|_{H^\gamma(\mathbb{T})}.
\end{equation}
\end{cor}

\begin{proof}[Proof]
It suffices to notice that
\[\textstyle\alpha >\frac{2}{3} \Longleftrightarrow \frac{1}{2}-\frac{\alpha}{4}<\frac{\alpha}{2},\]
so there exist $T_0>0$, only depending on $\|u_0\|_{H^{\alpha /2}}$, such that \eqref{mmt} can be solved locally in $X^{\alpha /2 , b}_\alpha(T_0)$. But the energy $\mathcal{H}_\alpha$ and the mass $Q$ are conserved along the trajectory, so $\|u(t)\|_{H^{\alpha /2}}$ remains bounded by $2(\mathcal{H}_\alpha+Q)(u_0)$. This proves that the solution is global, and we have $u(t)\in \mathcal{C}^\infty(\mathbb{T})$ for all $t\in \mathbb{R}$ because of the second part of the previous proposition. \eqref{moy} also follows.
\end{proof}

\subsection{End of the proof of theorem \ref{theommt}}\label{fin}
All the needed results are gathered : now we can study the growth of the Sobolev norms of the solutions of \eqref{mmt} for $\alpha\in (\frac{2}{3},1)$.

We begin with the $H^\alpha$-norm, introducing as in \eqref{energy} the modified energy :
\[\mathcal{E}_\alpha (u) := \|u\|_{L^2}^2 + \||D|^{\alpha}u\|_{L^2}^2
+ \underbrace{2\Re e(|D|^{\alpha}u, |u|^2u)}_{=:J_1(u)} -\underbrace{\textstyle\frac{1}{2} \displaystyle (|D|^{\alpha}(|u|^2),|u|^2)}_{=:J_2(u)},\]
when $u$ is a solution of \eqref{mmt}. $J_1$ and $J_2$ are of lower order than $\|u\|_{H^\alpha}^2$, so that when $\|u\|_{H^\alpha}$ is big enough, arguing as in section \ref{ozv}, we have $\frac{1}{2}\|u\|_{H^{\alpha}}^2 \leq \mathcal{E}_\alpha (u) \leq 2\|u\|_{H^{\alpha}}^2$.

Let us study the evolution of $\mathcal{E}_\alpha (u)$. The $L^2$-norm is conserved, so we directly pass on to
\[ \frac{d}{dt} \||D|^{\alpha}u\|_{L^2}^2 = 2\Re e (|D|^\alpha \dot{u},|D|^\alpha u) = -2\Re e(|D|^\alpha \dot{u},|u|^2u).\]
This combines with the derivative of $J_1(u)$, and gives rise to two terms :
\[ \frac{d}{dt} [\, \||D|^{\alpha}u\|_{L^2}^2+J_1(u)] = 2\Re e(|D|^\alpha u,\dot{u}|u|^2)+2\Re e(|D|^\alpha u,u(|u|^2)\dot{\:})=:Q_1(u)+Q_2(u).\]
Thanks to the equation, a simplification occurs : $Q_1(u)=-2\Im m (|D|^\alpha u,|u|^4u)$. The Sobolev embedding $H^{2/5}\hookrightarrow L^{10}$ and interpolation between $H^{\alpha /2}$ and $H^\alpha$ then allows to bound
\[ |Q_1(u)|\lesssim \|u\|_{H^\alpha}^{2-\varepsilon}, \quad\text{where }\varepsilon:=\tfrac{6}{\alpha}\left(\alpha-\tfrac{2}{3}\right) >0.\]
On the other hand, $Q_2(u)$ combines with the derivative of $J_2(u)$ :
\[Q_2(u)+\tfrac{d}{dt}J_2(u)=(\bar{u}|D|^\alpha u+u|D|^\alpha\bar{u}-|D|^\alpha(\bar{u}u),(|u|^2)\dot{\:}).\]
Since $(|u|^2)\dot{\:}=i(u|D|^\alpha\bar{u}-\bar{u}|D|^\alpha u)$, we bound $\|(|u|^2)\dot{\:}\|_{L^2}\lesssim \|u\|_{L^\infty}\|u\|_{H^\alpha}$. As for the other side of the scalar product, we appeal to lemma \ref{KPV}, since $\alpha <1$:
\[\|\bar{u}|D|^\alpha u+u|D|^\alpha\bar{u}-|D|^\alpha(\bar{u}u)\|_{L^2}\lesssim \||D|^\frac{\alpha}{2}u\|_{L^4}^2.\]
At this point, we have a Gagliardo-Nirenberg inequality :

\begin{lemme}\label{gn}
For any $p>2$, $s>0$, there exist $C>0$ such that
\[\||D|^sf\|_{L^p}\leq C\left( \|f\|_{L^\infty}+\left\||D|^{s\frac{p}{2}}f\right\|_{L^2}^\frac{2}{p}\left\| f\right\|_{L^\infty}^{1-\frac{2}{p}}\right) .\]
for every function $f:\mathbb{T}\to\mathbb{R}$.
\end{lemme}

Choose a real $\gamma \in (\frac{1}{2}-\frac{\alpha}{4},\frac{\alpha}{2})$, for instance $\gamma = \frac{2+\alpha}{8}$, and apply lemma \ref{gn} with $f=|D|^{\frac{\alpha}{2}-\gamma}u$, $p=4$ and $s=\gamma$. Thus
\[ \||D|^\frac{\alpha}{2}u\|_{L^4}^2 \lesssim \||D|^{\frac{\alpha}{2}+\gamma}u\|_{L^2}\||D|^{\frac{\alpha}{2}-\gamma}u\|_{L^\infty}\]
--- the other terms can be neglected.

All these calculations lead to the following fact : there exist a small $\theta >0$ (which can be chosen to be $\frac{3\alpha -2}{8\alpha}$), and constants $C_1,C_2>0$ such that for all $t\in \mathbb{R}$,
\[ \|u(t)\|_{H^\alpha}^2 \leq C_1\|u_0\|_{H^\alpha}^2 +C_2\int_0^t \|u(\tau )\|_{H^\alpha}^{2-2\theta}\|u(\tau )\|_{L^\infty}\||D|^{\frac{\alpha}{2}-\gamma}u(\tau )\|_{L^\infty}d\tau.\]
Denoting by $f(t)$ the right hand side of this inequality, and assuming that $t\geq 0$ without loss of generality, this implies that
\begin{align*}
f(t)^\theta-f(0)^\theta &= \int_0^t \frac{f'(\tau)d\tau}{f(\tau )^{1-\theta}} \\
&\leq C_2\int_0^t \|u(\tau )\|_{L^\infty}\||D|^{\frac{\alpha}{2}-\gamma}u(\tau )\|_{L^\infty}d\tau \\
&\leq C_2\sqrt{t}\cdot\|u\|_{L^4([0,t],L^\infty)}\||D|^{\frac{\alpha}{2}-\gamma}u\|_{L^4([0,t],L^\infty)} \\
&\leq C_2\sqrt{t}\left(\sum_{k=0}^{\lceil\frac{t}{T_0^{(\gamma )}}\rceil}\|u(\cdot-kT_0^{(\gamma)})\|_{X^{\gamma ,b}_\alpha(T_0^{(\gamma)})}\right)\left(\sum_{k=0}^{\lceil\frac{t}{T_0^{(\alpha /2)}}\rceil}\|u(\cdot-kT_0^{(\alpha /2)})\|_{X^{\frac{\alpha}{2} ,b}_\alpha(T_0^{(\alpha /2)})}\right) \\
&\leq C_2C_\gamma C_\frac{\alpha}{2}\|u_0\|_{H^\gamma}\|u_0\|_{H^{\frac{\alpha}{2}}}\sqrt{t}\left(\left\lceil \frac{t}{T_0^{(\gamma)}} \right\rceil +1\right) \cdot \left(\left\lceil \frac{t}{T_0^{(\alpha /2)}} \right\rceil +1\right),
\end{align*}
where we fixed a real $b\in (\frac{1}{2},1)$, and used the localized version of lemma \ref{linfty}, as well as \eqref{moy}. This achieves to show that the $H^\alpha$-norm of the solution of \eqref{mmt} grows at most polynomially, with the power of $t$ being less than $\frac{5}{4\theta}$, hence than $\frac{10 \alpha}{3\alpha -2}$.

\vspace{1em}
The end of the proof crucially relies on this first step. Indeed, to estimate the evolution of the $H^{\alpha +n}$-norm of $u$, with $n\geq 1$, we follow exactly the same scheme as for the proof of theorems \ref{theohw} and \ref{theommt} in section \ref{ozv}. Each time the $L^\infty$-norm of $u$ appears, we bound it by $\|u\|_{H^\alpha}$ (recall that $\alpha >\frac{1}{2}$). Besides, we do not interpolate the $H^s$-norms between $H^{\alpha /2}$ and $H^{\alpha +n}$ anymore, but between $H^{\alpha}$ and $H^{\alpha +n}$.

The only difference is that we need a new (and, to some extent, rougher) version of lemma \ref{leiblem} :

\begin{lemme}
Let $\frac{1}{2}<\alpha < 1$. For any integer $n\geq 1$, there is a constant $C_n>0$ (independent of $\alpha$) such that for all function $u\in H^{\alpha +n}(\mathbb{T})$,
\[ \|\bar{u}|D|^\alpha u+u|D|^\alpha \bar{u}-|D|^\alpha (\bar{u}u)\|_{H^{n}} \leq C \|u\|_{H^{\alpha}}^{1+\frac{1}{n}(\alpha -\frac{1}{2})}\|u\|_{H^{\alpha +n}} ^{1-\frac{1}{n}(\alpha -\frac{1}{2})}.\]
\end{lemme}

\begin{proof}[Proof]
Denote by $\mathcal{L}$ the left hand side of the inequality we intend to prove. Writing $u=\sum_{k\in \mathbb{Z}} u_ke^{ikx}$, we clearly have
\[ \begin{aligned} \mathcal{L}^2=&\sum_{k=-\infty}^{+\infty} |k|^{2n}\left| \sum_{l=-\infty}^{+\infty} (|l|^\alpha +|k-l|^\alpha -|k|^\alpha )u_l \overline{u_{l-k}} \right| ^2 \\
&\lesssim \sum_{k=-\infty}^{+\infty}\left[ \left| \sum_{l=-\infty}^{+\infty} |l|^{n}|k-l|^{\alpha}|u_l||\overline{u_{l-k}}|\right| ^2 +  \left| \sum_{l=-\infty}^{+\infty} |l|^{\alpha}|k-l|^{n}|u_l||\overline{u_{l-k}}|\right| ^2 \right], \end{aligned}\]
where we used the elementary inequality $|k|^n \leq 2^{n-1}(|l|^n+|k-l|^n)$ and the triangle inequality associated to the concave function $x\mapsto x^\alpha$ on $\mathbb{R}_+$. Now, define $\tilde{u}:=\sum_{k\in \mathbb{Z}} |u_k|e^{ikx}$, so that $\mathcal{L}\lesssim \||D|^n\tilde{u}\cdot|D|^\alpha\overline{\tilde{u}}\|_{L^2} \lesssim \|u\|_{H^{n+1/4}}\|u\|_{H^{\alpha +1/4}}$. Interpolating these norms between $H^{\alpha}$ and $H^{\alpha +n}$ leads to the result.
\end{proof}

All of this proves that there exist a small $\theta ' >0$, and constants $C_1$, $C_2>0$ such that for all $t\in \mathbb{R}$,
\[ \|u(t)\|_{H^{\alpha +n}}^2 \leq C_1\|u_0\|_{H^{\alpha +n}}^2 +C_2\int_0^t \|u(\tau )\|_{H^{\alpha +n}}^{2-2\theta '}\|u(\tau )\|_{H^\alpha}^{2+2\theta '} d\tau.\]
This holds with $\theta '=\frac{1}{2n}(\alpha -\frac{1}{2})$. On the other hand, we know that for some $A >0$, $\|u(\tau )\|_{H^\alpha} \lesssim (1+|\tau |)^A$ for all $\tau \in \mathbb{R}$. By Osgood's lemma, theorem \ref{theommt} is then fully established.

\section*{Appendices}
\appendix

\section{Growth of Sobolev norms for the Szeg\H{o} equation : an elementary bound}\label{sz-quad}
Let $u_0\in \mathcal{C}^\infty (\mathbb{T})$ with only nonnegative frequencies (which we denote by $u_0\in \mathcal{C}_+^\infty(\mathbb{T})$), and consider $t\mapsto u(t,x)$ the solution of the cubic Szeg\H{o} equation \eqref{sz} starting from $u_0$ at time $t=0$ : $u$ satisfies $i\partial_t u=\Pi_+(|u|^2u)$, and for all $t\in \mathbb{R}$, $u(t)$ also belongs\footnote{All the claims in this section can be found in \cite{ann} and are proven there.} to $\mathcal{C}_+^\infty(\mathbb{T})$. The purpose of this section is to give an elementary proof of the following estimate, which is the equivalent of theorem \ref{theohw} :
\begin{prop}\label{sz-prop}
For all $n\in \mathbb{N}$, there exist positive constants $C$ and $B$ such that
\begin{equation}\label{sz-t2} \|u(t)\|_{H^{1+n}}\leq Ce^{B|t|^2}.
\end{equation}
$C$ is a constant depending on $n$ and $\|u_0\|_{H^{1+n}}$, whereas $B$ can be chosen equal to $B_n\|u_0\|^8_{H^{1/2}}$ (here, $B_n$ depends only on $n$, not on the considered solution).
\end{prop}

We recall here that, though \eqref{sz-t2} is not the best bound available, it is the best one we can prove without resorting to the Lax pair formalism. The proof below only involves the boundedness of trajectories in the space $H^{1/2}(\mathbb{T})$ (which is due to the conservation of mass and momentum). It also uses a standard fact about Hankel operators.
\begin{déf}
Let $v\in H_+^{1/2}(\mathbb{T})$ (\emph{i.e.} $v$ has vanishing negative frequencies). The Hankel operator of symbol $v$ is the following $\mathbb{C}$--antilinear operator :
\[ H_v : \left\lbrace \begin{aligned} L^2_+(\mathbb{T}) &\longrightarrow L^2_+(\mathbb{T}) \\ h &\longmapsto \Pi_+(v\bar{h}). \end{aligned}\right.\]
\end{déf}

\begin{prop}\label{Hankel}
For $h\in L^2_+(\mathbb{T})$, we have $\|H_v(h)\|_{L^2}\leq \|v\|_{H^{1/2}}\|h\|_{L^2}$.
\end{prop}

\begin{proof}[Proof]
Expand $h=\sum_{k\geq 0}h_ke^{ikx}$ and $v=\sum_{k\geq 0}v_ke^{ikx}$. With these notations,
\[ \Pi_+(v\bar{h})=\sum_{k\geq 0}e^{ikx}\left( \sum_{l\geq k}v_l\overline{h_{l-k}}\right) .\]
A simple application of the Cauchy-Schwarz inequality gives
\[ \|\Pi_+(v\bar{h})\|_{L^2}^2\leq \sum_{k\geq 0}\left( \sum_{l\geq 0}|v_{l+k}|^2\right)\left( \sum_{l\geq 0}|h_l|^2\right) = \|h\|_{L^2}^2 \sum_{\tilde{k}\geq 0} (1+\tilde{k})|v_{\tilde{k}}|^2,\]
which is the yielded result.
\end{proof}

Since the embedding $H^{1/2}(\mathbb{T}) \hookrightarrow L^\infty(\mathbb{T})$ fails, proposition \ref{Hankel} is an improvement of the standard $L^\infty$-$L^2$ estimate of a product in $L^2$, similarly to lemma \ref{leiblem}. It is the key of the

\begin{proof}[Proof of proposition \ref{sz-prop}]We compute, for $n\geq 0$,
\[ \frac{d}{dt}\|\partial_x^{1+n}u\|_{L^2}^2=2\Re e (\partial_x^{1+n}\dot{u},\partial_x^{1+n}u)=2\Im m(\partial_x^{1+n}(|u|^2u),\partial_x^{1+n}u).\]
The last equality comes from the equation, and the fact that the frequencies of $u$ are nonnegative (so we get rid of $\Pi_+$ here). Then, using Leibniz rule, we expand $\partial_x^{1+n}(|u|^2u)=|u|^2\partial_x^{1+n}u + \sum_{k=1}^n\binom{1+n}{k}(\partial_x^k|u|^2)(\partial_x^{1+n-k}u) + u\partial_x^{1+n}|u|^2$. The first term cancels because of the imaginary part, and the "crossed terms" are easily estimated thanks to Sobolev injections and interpolation inequalities, as in the proof of theorem \ref{theohw}. As for the last term, we have
\[ (u\partial_x^{1+n}|u|^2,\partial_x^{1+n}u)=(\Pi_+(u\partial_x^{1+n}|u|^2),\partial_x^{1+n}u)=(H_u(\partial_x^{1+n}|u|^2),\partial_x^{1+n}u),\]
because $|u|^2$ is real. So, by proposition \ref{Hankel}, and then inequality \eqref{brezg-init},
\[\frac{d}{dt}\|u\|_{H^{1+n}}^2 \lesssim \|u\|_{H^{1+n}}^2\|u\|_{H^{1/2}}\|u\|_{L^\infty} \lesssim \|u\|_{H^{1+n}}^2\|u\|_{H^{1/2}}^2\sqrt{\log (1+\|u\|_{H^{1+n}}^2)},\]
and the proof is complete, once we have recalled that $\|u\|_{H^{1/2}}\leq C_0\|u_0\|^2_{H^{1/2}}$.
\end{proof}

\section{Some comments on the threshold $\alpha = \frac{2}{3}$}\label{wellp}
In this section, we discuss the bound $\alpha =\frac{2}{3}$ that naturally appears in the proof of theorem \ref{theommt}.

The crucial point, in the proof above, is to know which Strichartz estimate we are able to establish. In particular, following the strategy of \cite{BGT}, we would like to know for which values of the parameter $\gamma$ the inequality
\begin{equation}\label{str44}
\|e^{-it|D|^\alpha}u(x)\|_{L^4\left( (0,1)_t, L^4({\mathbb{T}_x})\right)} \leq C\|u\|_{H^\gamma}
\end{equation}
holds, whenever $u\in H^\gamma (\mathbb{T})$. If true, \eqref{str44} would imply that equation \eqref{mmt} is well-posed in $H^{\alpha /2}$, provided that $2\gamma <\frac{\alpha}{2}$, and we could then adapt the arguments we developped in section \ref{fin} to prove that solutions are polynomially bounded.

Looking back at \eqref{Sloc}, and interpolating the $L^4((0,1)_t, L^\infty(\mathbb{T}_x))$ estimate we obtained with the trivial $L^\infty((0,1)_t, L^2(\mathbb{T}_x))$ one, we find a bound for $\|e^{-it|D|^\alpha}u(x)\|_{L^8\left( (0,1)_t, L^4({\mathbb{T}_x})\right)}$, so \eqref{str44} is proved with $\gamma = \frac{1}{4}-\frac{\alpha}{8}$. The condition $2\gamma < \frac{\alpha}{2}$ exactly means that $\alpha >\frac{2}{3}$.

However, a scaling heuristic suggests that the natural value of $\gamma$ should rather be $\gamma_0:=\frac{1}{4}-\frac{\alpha}{4}$. Here, the condition $2\gamma_0<\frac{\alpha}{2}$ would enable us to extend the conclusions of theorem \ref{theommt} until $\alpha =\frac{1}{2}$.

In \cite{Demir}, by a different method, Demirbas, Erdo\u{g}an and Tzirakis also prove a Strichartz estimate in the case $\alpha >1$, but their result corresponds to ours (notice that in their work, they use other notations : what they call $\alpha$ is in fact half ours).

So far, we don't know if \eqref{str44} can be proved with $\gamma =\gamma_0$. Usual counter-examples (such as localized functions) only confirm that the scaling exponent $\gamma_0$ is the best we can hope. Besides, the difficulty is not due to the particular framework of the torus, since when $\alpha <1$ the speed of propagation of the waves (\emph{i.e.} the group velocity) is finite.

\vspace{0,5cm}
\bibliography{mabiblio}
\bibliographystyle{plain}

\end{document}